\def\arxivVersion{}
\def\mytitle{Symbolic Optimal Control}

\def\myname{
\ifx\shortnames\undefined Gunther \fi
Reissig
and
\ifx\shortnames\undefined Matthias \fi
Rungger
}

\def\mykeywords{Discrete abstraction, optimal control, difference
  inclusion, nonlinear system, symbolic control, approximate dynamic programming%
\ifx\arxivVersion\undefined\relax\else;
MSC: Primary, 49M25;
Secondary, 93C10, 93C55, 93C73%
\fi}
\let\submissionNote=\relax
\let\websourceNote=\relax

\ifx\FinalVersion\undefined%
\ifx\arxivVersion\undefined%
\def\DraftVersion{}
\documentclass[a4paper]{IEEEtran}
\else
\def\submissionNote{This work has been accepted for publication in the
\emph{IEEE Trans. Automatic Control}. Please refer to
\url{http://dx.doi.org/10.1109/TAC.2018.2863178}
for the definite publication.
}
\def\websourceNote{To reference this work, please find a {Bib\TeX}
  entry at
\href{http://www.reiszig.de/gunther/pubs/i13absoc.html}{author's homepage}.
}
\documentclass[journal,a4paper,onecolumn,12pt,nofonttune]{IEEEtran}
\date{}
\fi
\else
\def\IEEEtheorems{}
\documentclass[letterpaper]{IEEEtran}
\fi
\ifx\DraftVersion\undefined\relax\else%
\usepackage{ulem}
\def\emph#1{\textit{#1}}
\renewcommand\sout{\bgroup\markoverwith
{\textcolor{red}{\rule[.5ex]{2pt}{1pt}}}\ULon}
\def\sout#1{\bgroup\markoverwith
{\textcolor{red}{\rule[.5ex]{2pt}{1pt}}}\ULon{#1}}

\fi

\usepackage{cite}
\usepackage{pgf}
\usepackage{psfrag}
\usepackage{xcolor}

\usepackage{tikz}
\usetikzlibrary{automata}
\usetikzlibrary{shapes.geometric}
\usetikzlibrary{calc,shapes,arrows}

\usepackage{algorithm}
\usepackage[noend]{algpseudocode}

\algnewcommand\algorithmicinput{\textbf{Input:}}
\algnewcommand\Input{\item[\algorithmicinput]}
\algnewcommand\algorithmicoutput{\textbf{Output:}}
\algnewcommand\Output{\item[\algorithmicoutput]}
\algnewcommand\algorithmicparameter{\textbf{Parameter:}}
\algnewcommand\Parameter{\item[\algorithmicparameter]}

\algrenewcommand\algorithmicindent{1.0em}%

\usepackage{multirow}

\usepackage{svn-multi}
\svnid{$Id: i13absoc.tex 1449 2018-07-19 09:27:31Z lf1agure $}
\usepackage[cmex10]{amsmath}
\let\ORGforeignlanguage\foreignlanguage
\def\foreignlanguage#1{\lowercase{\ORGforeignlanguage{#1}}}
\makeatletter
\def\MakeUppercase#1{#1}%
\def\markboth#1#2{\def\leftmark{\@IEEEcompsoconly{\sffamily}\MakeUppercase{#1}}%
\def\rightmark{\@IEEEcompsoconly{\sffamily}\MakeUppercase{#2}}}
\makeatother
\ifx\arxivVersion\undefined\relax\else%

\fi

\usepackage{GR/gunther2e}

\usepackage{enumitem}
\usepackage{nicefrac}

\ifx
\hypersetup\undefined\def\href#1#2{\texttt{#2}}
\else
\hypersetup{
pdftitle={\mytitle},
pdfauthor={Gunther Reissig, Matthias Rungger, http://www.reiszig.de/gunther/},%
pdfsubject={\submissionNote \websourceNote},
pdfkeywords={\mykeywords}
}
\fi

\def\emptyset{\varnothing}

\ifx\DraftVersion\undefined\relax\else%
\IfFileExists{ReplyToReviewers.tex}{%
\usepackage{multibib} 
\newcites{review}{References}}{}
\fi

\begin{document}
\makeatletter
\markboth{\def\shortnames{}\myname\hspace*{\fill}\def\\{ }\mytitle\hspace*{\fill}\@date\hspace*{\fill}\ifx\DraftVersion\undefined\relax\else\hspace*{\fill}svn: \svnrev\fi\hspace*{\fill}}%
{\def\shortnames{}\myname\hspace*{\fill}\def\\{ }\mytitle\hspace*{\fill}\@date\hspace*{\fill}\ifx\DraftVersion\undefined\relax\else\hspace*{\fill}svn: \svnrev\fi\hspace*{\fill}}%
\makeatother

\title{\mytitle}

\author{\myname%
\thanks{%
\ifx\DraftVersion\undefined\relax\else%
Corresponding author: %
\fi
G.~Reissig is with the
Bundeswehr University Munich,
Dept. Aerospace Eng.,
Chair of Control Eng. (LRT-15),
D-85577 Neubiberg (Munich),
Germany,
\ifx\DraftVersion\undefined%
\url{http://www.reiszig.de/gunther/}%
\else%
\url{gunther@reiszig.de}%
\fi%
}%
\thanks{%
M.~Rungger is with the Hybrid Control Systems Group at the Department of
Electrical and Computer Engineering at the Technical University of Munich, Germany%
\ifx\DraftVersion\undefined.\else%
, \url{matthias.rungger@tum.de}%
\fi
}%
\thanks{This work has been supported by the German Research Foundation
  (DFG) under grant no. RE 1249/4-1.%
\ifx\arxivVersion\undefined\relax\else%
{} \submissionNote{} \websourceNote%
\fi%
}
}

\maketitle

\begin{abstract}
\noindent
We present novel results on the solution of a class of leavable, undiscounted
optimal control problems in the minimax sense for nonlinear, continuous-state,
discrete-time plants. The problem class includes entry-(exit-)time problems as
well as minimum time, pursuit-evasion and reach-avoid games as special cases.
We utilize auxiliary optimal control problems (``abstractions'') to compute both
upper bounds of the value function, i.e., of the achievable closed-loop
performance, and symbolic feedback controllers realizing those bounds. The
abstractions are obtained from discretizing the problem data, and we prove that
the computed bounds and the performance of the symbolic controllers converge to
the value function as the discretization parameters approach zero. In
particular, if the optimal control problem is solvable on some compact subset of
the state space, and if the discretization parameters are sufficiently small,
then we obtain a symbolic feedback controller solving the problem on that subset. These results
do not assume the continuity of the value function or any problem data, and they
fully apply in the presence of hard state and control constraints.
\end{abstract}

\ifx\arxivVersion\undefined\else%
\fi

\section{Introduction}
\label{s:intro}

In this paper we present novel results on the solution of
optimal control problems, in which we follow a symbolic
synthesis approach \cite{Tabuada09,i14sym,BeltaYordanovGol17} and
utilize finite, auxiliary problems (``abstractions'')
obtained from discretizing the original problem data.
Our theory provides symbolic feedback controllers, and it culminates
in novel convergence and completeness results including the following:
If the optimal control problem is solvable on some compact subset of
the state space, and if the discretization parameters are sufficiently
small, then the obtained controller solves the problem on that subset.
More specifically,
we consider discrete-time control systems that are defined by difference
inclusions of the form
\begin{IEEEeqnarray}{c}\label{e:sys}
x(t+1) \in F( x(t), u(t) ),
\end{IEEEeqnarray}
where $x(t)\in X$ and $u(t)\in U$ represents the
\concept{state} and the \concept{input signal}, respectively. Typically, the sets $X$
and $U$ are uncountably infinite.
We use set-valued \concept{transition functions}
\mbox{$F \colon X\times U\rightrightarrows X$} to account for possible perturbations such as
actuator inaccuracies and modeling
uncertainties; see e.g.~\cite{i14sym}. 
The problem data also includes non-negative,
extended real-valued \concept{running} and \concept{terminal cost functions},
$g$ and $G$,
\begin{subequations}
\label{e:costfunction}
\begin{align}
\label{e:costfunction:running}
&g \colon X \times X \times U \to \mathbb{R}_+ \cup \{ \infty \},\\
\label{e:costfunction:terminal}
&G \colon X \to \mathbb{R}_+ \cup \{ \infty \},
\end{align}
\end{subequations}
where $\mathbb{R}_{+}$ denotes the set of non-negative reals.
As we demonstrate in Section \ref{s:Example}, infinite costs are useful to represent hard actuation and state constraints.

Given the aforementioned problem data, we investigate optimal control
problems where the evolution of the closed-loop must be stopped at
some finite, but not predetermined, time. At that point, the
\concept{total cost} is determined as the sum of the terminal cost and
the previously accumulated running costs.
We seek to synthesize
a feedback controller that minimizes, or
approximately minimizes, the total cost in the minimax (worst-case)
sense, in which the controller generates both an input signal for
the plant \ref{e:sys} and additionally a signal that determines the
stopping time. 
In particular, the considered optimal control problem is
\concept{leavable} as the controller is allowed to stop the evolution
of the closed-loop at any time \cite{MaitraSudderth96}. In contrast to
similar settings,
in our problem stopping is mandatory and not discretionary, and we
penalize non-stopping evolutions with infinite costs.
The problem class is formally defined in Section \ref{ss:ProblemDef} and
includes entry-(or exit-)time problems as well as
minimum time, pursuit-evasion and reach-avoid games as special cases.
Examples are given in Sections \ref{ss:SpecialCases} and \ref{s:Example}.

\emph{Outline of the Proposed Approach.}
We follow a symbolic synthesis approach
\cite{Tabuada09,i14sym,BeltaYordanovGol17}:
First, an \concept{abstraction}, i.e., a finite, auxiliary optimal
control problem, is constructed by discretizing the problem
data. Second, a controller solving the auxiliary problem is
synthesized, and third, the latter controller is refined to obtain a
controller for the original problem.
In this context, we label quantities and objects that are defined with
respect to the original and to the auxiliary optimal control problem
as \emph{concrete} and \emph{abstract}, respectively.

In our theory, abstractions shall be constructed so that the abstract
\emph{value function}, i.e., the best achievable performance of the
abstract closed-loop, provides an upper bound of the concrete value
function. Conforming to the correct-by-construction paradigm of the
symbolic approach, the theory also guarantees that the
\emph{closed-loop value function} associated with the abstract
controller, i.e., the worst-case performance of that controller used
in the abstract closed loop, provides an upper bound of the
closed-loop value function associated with the concrete controller.

Since even rather coarse discretizations of the problem data
may very well qualify as abstractions, the abstract
value function will provide a rather conservative bound on the
concrete value function, in general.
To resolve that issue,
we shall introduce a suitable
notion of \concept{conservatism} for abstractions, which is closely
related to the accuracy by which the problem data is
discretized. As our main results, we shall establish the
convergence of both of the aforementioned upper bounds to the concrete value
function as the conservatism of the abstraction approaches zero. In turn,
as we shall also show, our synthesis approach is complete in the
following sense: If the original optimal control problem is solvable
on a compact subset of the state space, then the obtained controller
solves the original problem on that subset whenever a sufficiently
precise abstraction is employed.

Our results do not assume the continuity of the value function or
any problem data, and they fully apply in the presence of hard state
and control constraints.
The resulting feedback controllers are memoryless, finitely
representable and symbolic, i.e., they require only quantized as
opposed to full state information.

\emph{Related Work.}
The symbolic synthesis scheme
has been applied to a variety of optimal control problems including
minimum time problems \cite{MazoTabuada10b,Girard11}, entry-time problems
\cite{deRooMazo13,BrouckeDiBenedettoDiGennaroSangiovanniVincentelli05} and
finite horizon problems \cite{TazakiImura12}. Optimality properties in
combination with regular language specifications are analyzed in
\cite{LeongPrabhakar16}.
The results in \cite{MazoTabuada10b,deRooMazo13} are based on
approximate alternating simulation relations. As discussed in detail
in \cite[Sec.~IV]{i14sym}, this leads to overly complex, dynamic
controllers which additionally require full state information. The
controllers synthesized in \cite{Girard11} also require full state
information. Moreover, while the works
\cite{BrouckeDiBenedettoDiGennaroSangiovanniVincentelli05,TazakiImura12,LeongPrabhakar16}
lead to arbitrarily close approximations of value functions,
the respective convergence results
do not account for perturbations
\cite{BrouckeDiBenedettoDiGennaroSangiovanniVincentelli05,TazakiImura12,LeongPrabhakar16},
do not apply in the presence of hard constraints and discontinuous
value functions
\cite{BrouckeDiBenedettoDiGennaroSangiovanniVincentelli05,TazakiImura12}, 
or require piecewise linear plant dynamics \cite{LeongPrabhakar16}.
Additionally, the approach in
\cite{BrouckeDiBenedettoDiGennaroSangiovanniVincentelli05} relies on
the ability to exactly determine first integrals of the plant
dynamics, and the one in \cite{LeongPrabhakar16}, on the ability to
verify a non-trivial property for an exact optimal solution (which is
assumed to exist).

Closely related to our approach is the numerical approximation of the value
function, which has a rich history and has been a major research focus since the early days of Dynamic Programming
\cite{BellmanDreyfus62}. Related convergence results for deterministic
finite and
infinite horizon optimal control problems can be found in
\label{review:item13:text}
\cite{BertsekasShreve96,Bertsekas13,JiangJiang14,Heydari16b,BokanowskiForcadelZidani10,FisacChenTomlinSastry15,MargellosLygeros13},
and for several classes of stochastic optimal control problems, in
\cite{KushnerDupuis92,BertsekasTsitsiklis96,DufourPrietoRumeau12,SaldiLinderYuksel17b}.
Convergence results for leavable deterministic optimal control problems (or
deterministic optimal stopping problems), as considered in this paper,
are presented in
\cite{KreisselmeierBirkholzer94,KordaHenrionJones16,BardiBottacinFalcone95,CardaliaguetQuincampoixSaintPierre99,GrueneJunge07,GrueneJunge08}.
The vast majority of works focus on the special cases of minimum time
\cite{BardiBottacinFalcone95,CardaliaguetQuincampoixSaintPierre99}
and \mbox{entry-(or exit-)}time problems
\cite{GrueneJunge07,GrueneJunge08}
or on discounted running costs \cite{KordaHenrionJones16}, or apply only
to continuous-time problems
\cite{KordaHenrionJones16,BardiBottacinFalcone95,CardaliaguetQuincampoixSaintPierre99}.
Additionally, these works do not account for perturbations
\cite{KreisselmeierBirkholzer94,KordaHenrionJones16},
or do not apply in the presence of hard constraints
\cite{KreisselmeierBirkholzer94,BardiBottacinFalcone95}
and discontinuous value functions
\cite{KreisselmeierBirkholzer94,GrueneJunge07}.
While the works
\cite{BardiBottacinFalcone95,CardaliaguetQuincampoixSaintPierre99,GrueneJunge08}
do account for discontinuous value functions,
the respective results do not lead to controllers whose closed-loop
performances arbitrarily closely approximate the value function.

Another line of related research originates from the extension of asymptotically optimal
sampling-based motion planing \cite{KaramanFrazzoli11} to kinodynamic planning
that takes nonlinear dynamics into account \cite{LiLittlefieldBekris16}. In
contrast to our approach, the goal is not to
synthesize optimal feedback controllers, but to find an open-loop
input signal that optimally steers the system from a
fixed initial state to fixed final state or final region. Consequently, perturbations cannot be considered.
In addition, the convergence results in \cite{LiLittlefieldBekris16}
are probabilistic and do not provide worst-case guarantees.

\emph{Summary of Contributions.}
In view of the preceding discussion, we summarize our contributions as
follows.
Firstly, we characterize the value function as the
maximal fixed point of an appropriately defined Dynamic Programming
operator. A detailed comparison with related results
is provided in Section \ref{s:OptimalityPrinciple}.
Secondly, we propose a correct-by-construction approach to synthesize
memoryless symbolic controllers requiring only quantized state information,
as well as guarantees in the form of upper bounds on the controllers'
worst-case performances, for general classes of plant dynamics and
cost functions (Section \ref{s:ComparisonRefinementAbstraction}).
Thirdly,
and most importantly, we establish powerful convergence and
completeness results (Section \ref{s:convergence}), which imply that
even in the presence of hard constraints and discontinuous value
functions, our method is capable of synthesizing controllers whose
performance guarantees arbitrarily closely approximate the best
achievable performance.
In Section \ref{s:Example}, we demonstrate
our approach on three examples.

For the sake of self-consistency of the paper, we present in Section
\ref{s:AlgorithmicSolution} our method from \cite{i17conv} to compute
abstractions for a class of sampled control systems, and we also
present an algorithm to efficiently solve auxiliary, abstract optimal
control problems.
In the Appendix we collect some auxiliary results numbered
\ref{prop:hypoLimit:i13absoc} through \ref{cor:uscCompact}.
Preliminary versions of some of the results in this paper have been
announced
in \cite{i13absocc}.

\section{Preliminaries}
\label{s:prelims}
The relative complement of the set $A$ in the set $B$ is denoted by
$B \setminus A$.
$\mathbb{R}$, $\mathbb{R}_+$, $\mathbb{Z}$ and $\mathbb{Z}_{+}$ denote the sets of
real numbers, non-negative real numbers, integers and non-negative integers, respectively,
and $\mathbb{N} = \mathbb{Z}_{+} \setminus \{ 0 \}$. We adopt the
convention that $\pm \infty + x = \pm \infty$ for any
$x \in \mathbb{R}$.
$\intcc{a,b}$, $\intoo{a,b}$,
$\intco{a,b}$, and $\intoc{a,b}$
denote closed, open and half-open, respectively,
intervals with end points $a$ and $b$,
e.g.~$\intco{0,\infty} = \mathbb{R}_{+}$.
$\intcc{a;b}$, $\intoo{a;b}$,
$\intco{a;b}$, and $\intoc{a;b}$ stand for discrete intervals,
e.g.~$\intcc{a;b} = \intcc{a,b} \cap \mathbb{Z}$,
$\intco{1;4} = \{ 1,2,3\}$, and
$\intco{0;0} = \emptyset$.
$\max M$, $\min M$, $\sup M$ and $\inf M$ denote the
maximum, the minimum, the supremum and the infimum,
respectively, of the nonempty subset
$M \subseteq \intcc{-\infty,\infty}$, and
we adopt the convention that $\sup \emptyset = 0$.

$f \colon A \rightrightarrows B$ denotes a \concept{set-valued map}
from the set $A$ into the set $B$, whereas $f \colon A \to B$ denotes
an ordinary map; see \cite{RockafellarWets09}.
The set of maps $A \to B$ is denoted $B^A$.
If $f$ is set-valued, then $f$ is \concept{strict} and
\concept{single-valued} if $f(a) \not= \emptyset$ and $f(a)$ is a
singleton, respectively, for every $a$.

We identify set-valued maps $f \colon A \rightrightarrows B$ with
binary relations on $A \times B$, i.e.,
$(a,b) \in f$ iff $b \in f(a)$.
Moreover, if $f$ is single-valued, it
is identified with an ordinary map $f \colon A \to B$.
The restriction of $f$ to a subset $M \subseteq A$ is denoted $f|_{M}$.
The inverse mapping $f^{-1} \colon B \rightrightarrows A$ is defined
by $f^{-1}(b) = \Menge{a \in A}{b \in f(a)}$, $f \circ g$ denotes
the composition of $f$ and $g$, $(f \circ g)(x) = f(g(x))$, and the
image of a subset $C \subseteq A$ under $f$ is denoted
$f(C)$, $f(C) = \bigcup_{a \in C} f(a)$.

If $A$ and $B$ are metric spaces, then $f$ is
\concept{upper semi-continuous} (u.s.c.) if
$f^{-1}(\Omega)$ is closed for every
closed subset $\Omega \subseteq B$.
Alternatively, if $B = \intcc{-\infty,\infty}$, then $f$ is \concept{bounded} on the
subset $C \subseteq A$ if $f(C)$ is a bounded subset of $\mathbb{R}$.

For maps $f, g \colon X \to \intcc{-\infty,\infty}$, the relations $<$,
$\le$, $\ge$, $>$ are defined point-wise, e.g. $f < g$ if
$f(x) < g(x)$ for all $x \in X$. Analogously, the relations are
interpreted component-wise for elements
of~$\intcc{-\infty,\infty}^n$.
The set of minimum points of $f$ in some subset $ Q\subseteq X$ is
denoted $\argmin\Menge{f(x)}{x \in Q}$.
$\hypo f = \Menge{(x,\gamma) \in X \times \mathbb{R}}{\gamma \le f(x)}$
is the \concept{hypograph} of $f$, and $f$ is u.s.c. if $X$ is a metric space
and $\hypo f \subseteq X \times \mathbb{R}$ is closed
\cite{RockafellarWets09,HuPapageorgiou97.i}.
The \concept{backward shift operator} $\sigma$ is defined as follows.
If the map $f$ is defined on $\intco{0;T}$ for some
$T \in \mathbb{N} \cup \{\infty\}$, then $\sigma f$ is the map defined
on $\intco{0;T-1}$ and given by $(\sigma f)(t) = f(t+1)$.

\section{A Leavable Optimal Control Problem}
\label{s:problem}

We develop our theory in a
rather general setting, and for now we simply assume that $X$ and $U$ are
nonempty sets.
These assumptions already allow for a fixed-point characterization of the value
function.
As we progress with our analysis we gradually impose stricter
assumptions. In particular, we demonstrate the upper semi-continuity
of the value function under assumptions including that $X$ and $U$ are
metric spaces. Here the abstract treatment of
$X$ and $U$ is crucial.
Even if the original system evolves in $\mathbb{R}^n$,
the abstractions we shall construct do not.
Similarly, to prove our main results in Section \ref{s:convergence},
we will need to construct yet another auxiliary problem with a
non-euclidean state alphabet.
Moreover,
our setting covers plants whose states naturally form
finite-dimensional manifolds, which is common in e.g.~robot dynamics
\cite{ChosetLynchHutchinsonKantorBurgardKavrakiThrun05}.

\subsection{Problem definition}
\label{ss:ProblemDef}

We seek to control \concept{systems} whose dynamics is defined by
difference inclusions of the form \ref{e:sys}. Subsequently, we often
refer to these systems as \concept{plants}.
Controllers, on the other hand, are defined by more general inclusions
of the form 
\begin{equation}
\label{e:controller}
( z(t+1), u(t), v(t) ) \in H( z(t), x(t) ),
\end{equation}
where $z$ represents the state of the controller.
The controller accepts a state signal $x$ of the plant as its input
and generates a signal $u$ that serves as input for the plant.
See \ref{fig:ClosedLoopSimpleAndMealy}.
The controller additionally generates a stopping signal $v$ which is
used to terminate the evolution of the closed loop and will be
discussed in conjunction with our definition of cost functionals.
We formalize the aforementioned concepts below.

\begin{definition}
\label{def:System}
A \concept{system} is a triple
\begin{equation}
\label{e:def:System}
(X,U,F),
\end{equation}
where $X$ and $U$ are nonempty sets and
$F \colon X \times U \rightrightarrows X$ is strict.
A pair
$(u,x) \in U^{\mathbb{Z}_{+}} \times X^{\mathbb{Z}_{+}}$
is a \concept{solution} of the system \ref{e:def:System}
if \ref{e:sys} holds for all $t \in \mathbb{Z}_{+}$.

A \concept{controller} for the system \ref{e:def:System} is a
quintuple
\begin{equation}
\label{e:def:Controller}
(Z,Z_0,\widetilde{X},\widetilde{U},H),
\end{equation}
where
$Z$, $Z_0$, $\widetilde{X}$, $\widetilde{U}$
are non-empty sets,
$Z_0 \subseteq Z$, $X \subseteq \widetilde{X}$, $\widetilde{U} \subseteq U$,
and
$H \colon Z \times \widetilde{X} \rightrightarrows Z \times \widetilde{U} \times \{0,1\}$
is strict.
A controller \ref{e:def:Controller} is \concept{static} if $Z$ is a singleton.
A quadruple
$(u,v,z,x)
\in
\widetilde{U}^{\mathbb{Z}_{+}}
\times
\{0,1\}^{\mathbb{Z}_{+}}
\times
Z^{\mathbb{Z}_{+}}
\times
\widetilde{X}^{\mathbb{Z}_{+}}
$
is a \concept{solution} of the controller \ref{e:def:Controller}
if $z(0) \in Z_0$
and \ref{e:controller} holds for all $t \in \mathbb{Z}_{+}$.
\end{definition}

We use $C \in \mathcal{F}(X,U)$ to denote the fact that $C$ is a
controller for the system \ref{e:def:System}.
The sets $X$ and $Z$ are the \concept{state alphabet},
$Z_0$ is the \concept{initial state alphabet},
$U$ and $\widetilde{X}$ are the \concept{input alphabet},
and
the maps $F$ and $H$ are the \concept{transition function}, of the
system \ref{e:def:System} and the controller \ref{e:def:Controller},
respectively.

We emphasize that our notion of controller
is equivalent to the respective notion in \cite{i14sym} in the
non-blocking case,
and it subsumes related notions from the literature, such as
\concept{causal feedback strategy}~\cite[Ch.~VIII]{BardiCapuzzoDolcetta97},
\concept{control strategy} \cite{CooganArcakBelta17},
\concept{feedback plan} \cite{LaValle06},
and \concept{policy}~\cite{BertsekasShreve96}.
Specifically, any strict policy
$\mu \colon Z \times X \rightrightarrows U \times \{0,1\}$
with
$
Z
=
\bigcup_{T \in \mathbb{Z}_{+}}
U^{\intco{0;T}} \times X^{\intco{0;T}}
$,
which generates signals $u$ and $v$ according to
\[
( u(t), v(t) )
\in
\mu( u|_{\intco{0;t}}, x|_{\intco{0;t}}, x(t) )
\]
in place of \ref{e:controller}, can be equivalently represented by a
controller with state alphabet $Z$.
On the other hand, as we shall see later, static (or \concept{memoryless})
controllers are
sufficient to approximately solve the optimal control
problems investigated in the present paper, to arbitrary accuracy.

\begin{definition}
\label{def:Behavior}
Let $S$ denote the system \ref{e:def:System} and suppose that
$C \in \mathcal{F}(X,U)$, where $C$ is of the form
\ref{e:def:Controller}.

The \concept{behavior}
$
\mathcal{B}(C \times S)
\subseteq
( U \times \{ 0, 1 \} \times X )^{\mathbb{Z}_{+}}
$
of the closed-loop composed of $C$ and $S$ is defined by the
requirement that $(u,v,x) \in \mathcal{B}(C \times S)$ iff
there exists
a signal $z \colon \mathbb{Z}_{+} \to Z$
such that
$(u,v,z,x)$
is a solution of $C$ and
$(u, x)$
is a solution of $S$.
In addition, the \concept{behavior initialized at $p \in X$} is
denoted by $\mathcal{B}_p(C \times S)$ and defined by
$
\mathcal{B}_p(C \times S)
=
\Menge{ (u,v,x) \in \mathcal{B}(C \times S)}{x(0)=p}
$.
\end{definition}
\begin{figure}[t]
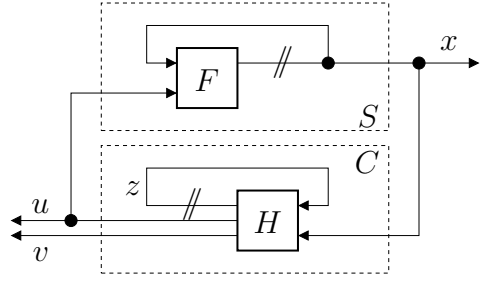

\begin{center}
\psfrag{F}[][]{$F$}
\psfrag{H}[][]{$H$}
\psfrag{x}[][]{$x$}
\psfrag{v}[][]{$v$}
\psfrag{u}[][]{$u$}
\psfrag{z}[r][r]{$z$}
\psfrag{C}[r][r]{$C$}
\psfrag{S}[r][r]{$S$}
\psfrag{//}[][]{$\sslash$}
\pgfdeclareimage[width=\ifCLASSOPTIONonecolumn.343\else.7\fi\linewidth]{ClosedLoopSimpleAndMealy}{figures/ClosedLoop.MooreMealy}%
\noindent
\pgfuseimage{ClosedLoopSimpleAndMealy}
\end{center}
\caption{\label{fig:ClosedLoopSimpleAndMealy}
Closed loop $C \times S$ according to
Definition \ref{def:Behavior}. The symbol
$\sslash$ denotes a delay.
}
\end{figure}

Our problem data also includes
a \concept{running cost function} $g$ and a
\concept{terminal cost function} $G$ as in \ref{e:costfunction}.
The total cost to be minimized is then given by the
\concept{cost functional}
$J \colon( U \times \{ 0, 1 \} \times X)^{\mathbb{Z}_+} \to \intcc{0,\infty}$,
which is defined as the sum of the terminal cost and accumulated
running costs, i.e.,
\begin{subequations}
\label{e:cfunctional}
\begin{align}
\label{e:cfunctional:a}
J(u,v,x)
&=
G(x(T))+
\sum_{t=0}^{T-1}g(x(t),x(t+1),u(t))\\
\intertext{\noindent if $v\neq 0$ and $T=\min v^{-1}(1)$, and otherwise we define $J$ by}
J(u,v,x)
&=
\infty.
\end{align}
\end{subequations}

Throughout the paper, we identify the optimal control problem with its
problem data and use the following definition.

\begin{definition}\label{d:OST}
An \concept{optimal control problem} is a tuple
\begin{IEEEeqnarray}{c}\label{e:OCP}
(X,U,F,G,g),
\end{IEEEeqnarray}
where \ref{e:def:System} is a system and $G$ and $g$ are
non-negative extended real-valued functions as
in~\ref{e:costfunction}.
\end{definition}
The notions of state alphabet, input alphabet and transition function
are carried over from the system \ref{e:def:System} to the optimal
control problem \ref{e:OCP} in the obvious way.

As already indicated, solving the problem \ref{e:OCP}
means to find controllers $C \in \mathcal{F}(X,U)$
which, for every state $p \in X$, minimize or approximately minimize
the cost \ref{e:cfunctional} for
$(u,v,x) \in \mathcal{B}_p(C\times S)$ in a worst-case sense, where
$S$ denotes the plant \ref{e:def:System}.
Here, the stopping signal $v \colon \mathbb{Z}_+ \to \{ 0, 1 \}$
determines, by its first $0$-$1$ edge, the instance of time when the
optimization process stops and the terminal costs are evaluated, and
the worst-case cost is given by the
\concept{closed-loop value function}
$L \colon X \to \intcc{0,\infty}$
of \ref{e:OCP} associated with $C$,
\begin{IEEEeqnarray}{c}
\label{e:ClosedLoopValueFunction:L}
L(p)
=
\sup_{(u,v,x)\in \mathcal{B}_{p}( C \times S )}
J(u,v,x).
\end{IEEEeqnarray}
It follows that the achievable closed-loop performance is determined
by the \concept{value function}
$V \colon X \to \intcc{0,\infty}$
of \ref{e:OCP},
\begin{equation}
\label{e:TimeDiscreteValueFunction:V}
V(p) =
\inf_{C \in \mathcal{F}(X,U)}
\sup_{(u,v,x) \in \mathcal{B}_{p}(C \times S)}
J(u,v,x).
\end{equation}
As we show in Theorem~\ref{th:OptimalityPrinciple}, the value
function satisfies 
\begin{equation}
\label{th:OptimalityPrinciples:e:1}
V(p)
=
\sup_{\beta \in \Delta(p)}
\;
\inf_{u \in U^{\mathbb{Z}_{+}}}
\;
\inf_{v \in \{ 0, 1 \}^{\mathbb{Z}_{+}}}
J(u,v,\beta(u))
\end{equation}
for all $p \in X$, where
$\Delta(p)$ is the set of all strictly causal maps
$
\beta \colon U^{\mathbb{Z}_{+}}
\to
X^{\mathbb{Z}_{+}}
$
for which
$(u,\beta(u))$
is a solution of $S$
satisfying $\beta(u)(0) = p$,
for every $u \in U^{\mathbb{Z}_{+}}$.
Here, $\beta$ is \concept{causal} (resp., \concept{strictly causal}) if
$\beta(u)|_{\intcc{0;T}} = \beta(\tilde u)|_{\intcc{0;T}}$ whenever
$u, \tilde u \in U^{\mathbb{Z}_{+}}$,
$T \in \mathbb{Z}_{+}$
and
$u|_{\intcc{0;T}} = \tilde u|_{\intcc{0;T}}$ (resp.,
$u|_{\intco{0;T}} = \tilde u|_{\intco{0;T}}$).
Thus, in terms of performance, our concept of controller is equivalent
to \concept{non-anticipating strategies}~\cite{GrueneJunge08}.

\subsection{Important Special Cases}
\label{ss:SpecialCases}

We briefly discuss special cases of the class of optimal
control problems considered in this paper. For further examples,
including an entry-(or exit-)time problem and a problem whose
underlying dynamics is chaotic, see Section \ref{s:Example}.

\begin{example}[Shortest Path Problem]
\label{ex:ShortestPaths}
Given a directed graph, we wish to find shortest paths from a specified source vertex to
all other vertices \cite{AhujaMagnantiOrlin93}. This problem and its
generalizations have numerous applications
\cite{Knuth77,AhujaMagnantiOrlin93,GalloLongoPallotinoNguyen93,GrueneJunge08}.
For a formal description,
let $X$ and $A \subseteq X \times X$ be finite sets of
vertices and of arcs, respectively, of a directed graph, and let
$s \in X$ be a distinguished source vertex. Let a non-negative
length $w_{p,q}$ be associated with each arc $(p,q)$, i.e.,
$w \colon A \to \mathbb{R}_{+}$, and define the length of a path as
the sum of the lengths of its arcs.
The distance from $s$ to $p \in X$, denoted $d(p)$, is the minimum
length of any (directed) path from $s$ to $p$, and is defined to be
$\infty$ if no such path exists. The problem is to determine $d(p)$,
and a path of length $d(p)$ from $s$ to $p$ if $d(p) < \infty$, for
all $p \in X$.

The problem can be reduced to the following instance of the optimal control
problem \ref{e:OCP}.
Define
$U = X$,
$G(s) = 0$, and
$G(p) = \infty$ for all $p \in X \setminus \{ s \}$,
let $g$ be such that $g(p,q,u) = w_{q,p}$ whenever $(q,p) \in A$, and
let $F$ be single-valued such that
$F(p,U) = \{ p \} \cup \Menge{ q \in X}{ (q,p) \in A }$.
Then there exists a static controller $C$
for the system $S$ in \ref{e:def:System},
with single-valued
transition
function,
such that the closed-loop value
function of \ref{e:OCP} associated with $C$ equals the value function
$V$ of \ref{e:OCP}; see e.g.~Section \ref{ss:AlgorithmicSolution:TheAbstractController}.
In turn, as is easily seen, a shortest path from $s$ to $p$ can be
obtained from the unique element of $\mathcal{B}_p(C \times S)$, and
$d = V$.
\end{example}

\begin{example}[Reach-Avoid Problem]
\label{ex:ReachAvoid}
The problem of steering the state of the plant into a target set while
avoiding obstacles appears in many applications,
e.g.~\cite{Isaacs65,ChosetLynchHutchinsonKantorBurgardKavrakiThrun05}. For
a formal description, let $S$ be a plant of the form
\ref{e:def:System}, and let a target set $D \subseteq X$ and an
obstacle set $M \subseteq X$ be given. The controller
$C \in \mathcal{F}(X,U)$ is successful for the state $p \in X$ if
for every $(u,v,x) \in \mathcal{B}_p(C \times S)$ there exists some
$s \in \mathbb{Z}_{+}$ satisfying $x(s) \in D$ and
$x(t) \not\in M$ for all $t \in \intcc{0;s}$. We say that $p$ can be
forced into the target set if there exists a controller that is
successful for $p$. The problem is to determine the subset
$E \subseteq X$ of states that can be forced into the target set, and
a controller that is successful for all states in $E$.

The problem can be reduced to the following instance of the optimal
control problem \ref{e:OCP}.
Define $G(p) = 0$ if $p \in D \setminus M$,
and otherwise $G(p) = \infty$, and define
$g(p,q,u) = 0$ if $p \not\in M$,
and otherwise
$g(p,q,u) = \infty$.
Then $E$ equals the effective domain
$V^{-1}(\mathbb{R_+})$ of the value function $V$ of \ref{e:OCP}, and a
controller $C$ is successful for all states in $E$ iff the closed-loop
value function of \ref{e:OCP} associated with $C$
equals $V$.

The problem can be approximately solved using the results in this
paper, which, for each compact subset $K \subseteq E$, yield a
static controller $C$ that is successful for all states in $K$.
See Corollary \ref{cor:th:QualitativeCompleteness}.
\end{example}

\begin{example}[Minimum Time Problem]
\label{ex:MinTime}
Various practical problems require solving reach-avoid problems in
minimum time,
e.g.~\cite{BrysonHo75,CardaliaguetQuincampoixSaintPierre99,ChosetLynchHutchinsonKantorBurgardKavrakiThrun05}. For
a formal description, let $S$, $D$ and $M$ as in Example
\ref{ex:ReachAvoid}, and define the \concept{entry time $T_C(p)$ from
  $p \in X$ under feedback $C \in \mathcal{F}(X,U)$} as the infimum of
all $\tau \in \mathbb{Z}_{+}$ satisfying the following condition:
For every $(u,v,x) \in \mathcal{B}_p(C \times S)$ there exists some
$s \in \intcc{0;\tau}$ such that $x(s) \in D$ and
$x(t) \not\in M$ for all $t \in \intcc{0;s}$. The
\concept{entry time $T(p)$ from $p \in X$} is the infimum of $T_C(p)$
over all controllers $C \in \mathcal{F}(X,U)$. The problem is to
determine the value $T(p)$ for all $p \in X$, and a controller
$C \in \mathcal{F}(X,U)$ satisfying $T = T_C$.

The problem can be reduced to the following instance of the optimal
control problem \ref{e:OCP}.
Define $G(p) = 0$ if $p \in D \setminus M$,
and otherwise $G(p) = \infty$, and define
$g(p,q,u) = 1$ if $p \not\in M$,
and otherwise
$g(p,q,u) = \infty$.
Then the \concept{minimum time function} $T$
equals the value function $V$ of \ref{e:OCP}, and for every
controller $C \in \mathcal{F}(X,U)$, $T_C$ equals the closed-loop
value function of \ref{e:OCP} associated with $C$.

The problem can be approximately solved using the results in this
paper, which, for each compact subset $K \subseteq X$,
yield a static controller $C$ satisfying
$\sup T_C(K) = \sup T(K)$.
See Corollary \ref{cor:th:QualitativeCompleteness}.
\end{example}

\section{Fixed-point characterization and\\regularity of the value function}
\label{s:OptimalityPrinciple}

In this section, we shall first characterize the value function
\ref{e:TimeDiscreteValueFunction:V} as the maximal fixed-point of the
\concept{dynamic programming operator} $P$ associated with the optimal
control problem \ref{e:OCP},
\begin{equation}
\label{e:TimeDiscreteDPOperator:P}
P(W)(p)
=
\min
\left\{
G(p),
\inf_{u \in U}
\sup_{q \in F(p,u)}
g(p,q,u) + W(q)
\right\},
\end{equation}
which maps the space of functions
$X \to \intcc{0,\infty}$ into itself. This
characterization will in turn permit us to represent the value
function as the limit of repeated applications of $P$ to the terminal
cost function and to prove that this limit is semi-continuous.
These results are new, see our discussion at the end of this
section. Moreover, they will be useful later, when they facilitate the
comparison of value functions in Section
\ref{s:ComparisonRefinementAbstraction} as well
as our
\label{review:item6:text:a}
convergence proofs in Section \ref{s:convergence}.
In addition, as a side product we obtain the identity
\ref{th:OptimalityPrinciples:e:1}, which shows that in our setting,
the value function could equivalently be defined using alternative
information patterns, e.g.~\cite{GrueneJunge08}.

\begin{theorem}
\label{th:OptimalityPrinciple}
Let \ref{e:OCP} be an optimal control problem, and let $V$ and $P$ be
the associated value function and dynamic programming operator as
defined in \ref{e:TimeDiscreteValueFunction:V} and
\ref{e:TimeDiscreteDPOperator:P}, respectively.
Then $V$ is the maximal fixed point of $P$, i.e.,
$P(V) = V$, and
$W \le P(W)$ implies $W \le V$.
Moreover, the identity \ref{th:OptimalityPrinciples:e:1} holds for all
$p \in X$.
\end{theorem}

\begin{proof}
We first observe that
$P$ is monotone, i.e., that $P(W) \le P(W')$ whenever
$W \le W'$, and that
\begin{alignat*}{3}
\negthinspace
J(u,v,x)
&=
G(x(0))
&\;&
\text{if $v(0) = 1$,}\\
\negthinspace
J(u,v,x)
&=
g(x(0),x(1),u(0)) + J(\sigma u, \sigma v, \sigma x),
&&
\text{otherwise,}
\end{alignat*}
for all
$(u,v,x) \in ( U \times \{ 0, 1 \} \times X )^{\mathbb{Z}_{+}}$.
Using a controller whose
transition
function maps into
$Z \times U \times \{1\}$, for some $Z$,
we see that $V \le G$.

In what follows, we shall denote by $R(p)$ the right hand side of
\ref{th:OptimalityPrinciples:e:1} to show that
$R \le V \le P(V)$ and that $W \le P(W)$ implies $W \le R$, which
proves the theorem. In particular, the case $W = P(V)$ shows that $P(V) \le V$.

To prove that $R \le V$ holds, assume that $V(p) < R(p)$ for some
$p \in X$. Then there exists a controller $C$ of the form
\ref{e:def:Controller} and a map $\beta \in \Delta(p)$ satisfying
$J(u,v,x) < J(u,v,\beta(u))$ for every
$(u,v,x) \in \mathcal{B}_p(C \times S)$,
where $S$ denotes the system \ref{e:def:System}.
We will inductively construct
$u$ and $v$ such that 
$(u,v,\beta(u)) \in \mathcal{B}_p(C \times S)$, which is a
contradiction and so proves $R \le V$.
To this end, consider the following condition for any
$T \in \mathbb{Z}_+$:
The signals $u$ and $v$ have already been defined on
$\intco{0;T}$, and the signal $z$ has already been defined on
$\intcc{0;T}$ such that
\ref{e:controller} with $\beta(u)$ in place of $x$
holds for all $t \in \intco{0;T}$. Here, $\beta(u)(t)$ denotes
$\beta(\tilde u)(t)$ for any extension
$\tilde u \colon \mathbb{Z}_+ \to U$ of $u$, which is an unambiguous
abbreviation as $\beta$ is causal.
Pick any $z(0) \in Z_0$ to satisfy the condition for $T = 0$, and
assume the condition holds for some $T \in \mathbb{Z}_+$. To extend
the signals $u$, $v$, and $z$ we
pick
$( z(T+1), u(T), v(T) )
\in
H( z(T), \beta(u)(T) )
$,
which is feasible as $\beta$ is strictly causal.
Then the condition holds with $T+1$ in
place of $T$ as $\beta$ is causal. Consequently, there exist signals
$u$, $v$ and $z$ defined on $\mathbb{Z}_+$ such that
$(u,v,z,\beta(u))$
is a solution of $C$, and so
$(u,v,\beta(u)) \in \mathcal{B}_p(C \times S)$ as
$\beta \in \Delta(p)$.

To prove that $V \le P(V)$ holds, it suffices to show that
\begin{equation}
\label{th:OptimalityPrinciples:proof:e:1}
V(p)
\le
\sup_{q \in F(p,\xi)}
g(p,q,\xi) + V(q)
\end{equation}
for all $p \in X$ and all $\xi \in U$.
To this end, let $p \in X$, $\xi \in U$ and $\varepsilon > 0$.
For every $q \in F(p,\xi)$ there is a controller
$C_q \in \mathcal{F}(X,U)$ such that
\begin{equation}
\label{th:OptimalityPrinciples:proof:e:2}
\sup_{(u,v,x) \in \mathcal{B}_{q}(C_q \times S)}
J(u,v,x)
\le
V(q) + \varepsilon.
\end{equation}
We may assume without loss of generality that $C_q$ is of the form
$
C_q
=
(Z_q,\{ z_{q,0} \}, X, U, H_q),
$
in which the state alphabets $Z_q$
are pairwise disjoint.
Let $z_0,z_1 \not\in Z_q$ for every $q$, $z_0 \not= z_1$, define
$Z = \{ z_0, z_1 \} \cup \bigcup_{q \in F(p,\xi)} Z_q$, and
let $\mu$ be a controller for $S$ of the form
$(Z, \{ z_0 \}, X, U, H)$
that satisfies the following conditions
for every $q \in F(p,\xi)$:
$H(z_0,p) = \{( z_1, \xi, 0 )\}$,
$H(z_1,q) = H_q(z_{q,0},q)$, and
$H(z,\cdot) = H_q(z,\cdot)$
whenever $z \in Z_q$.
One easily shows that
$(u,v,x) \in \mathcal{B}_{p}(\mu \times S)$ implies
$x(0) = p$,
$v(0) = 0$,
$u(0) = \xi$,
$x(1) \in F(p,\xi)$, and
$(\sigma u,\sigma v,\sigma x) \in \mathcal{B}_{x(1)}(C_{x(1)} \times S)$.
Using the definition of $V$,
the observation at the beginning of this proof,
and \ref{th:OptimalityPrinciples:proof:e:2}, it then follows that
$V(p)
\le
\sup_{q \in F(p,\xi)}
g(p,q,\xi) + V(q) + \varepsilon
$.
This implies \ref{th:OptimalityPrinciples:proof:e:1},
and so $V \le P(V)$.

Finally, suppose that $W \le P(W)$ and that
$R(p) + 2 \varepsilon < W(p)$ for some $p \in X$ and some
$\varepsilon > 0$. We claim that there exists a map
$\beta \in \Delta(p)$ such that
\begin{equation}
\label{th:OptimalityPrinciples:proof:e:3}
R(p)
+
(1+2^{-t}) \varepsilon
<
W( \beta(u)(t) )
+
\Sigma(\beta(u), u, t)
\end{equation}
holds for all
$t \in \mathbb{Z}_{+}$ and all $u \colon \mathbb{Z}_{+} \to U$,
where $\Sigma$ is defined by
$
\Sigma(x,u,t)
=
\sum_{\tau = 0}^{t-1} g( x(\tau), x(\tau + 1), u(\tau) ).
$
Since $W \le P(W) \le G$ it then follows that
$R(p) + \varepsilon \le J( u, v, \beta(u) )$ for all $u$ and $v$,
which contradicts the definition of $R$, and hence, shows that
$W \le P(W)$ implies $W \le R$.

To prove our claim, we define $\beta(u)(0) = p$ for every $u$, so that
the inequality \ref{th:OptimalityPrinciples:proof:e:3} for $t=0$
reduces to our assumption on $\varepsilon$. Next, we assume that for
some $t \in \mathbb{Z}_{+}$ and all $\tau \in \intcc{0;t}$,
the value of $\beta(u)|_{\intcc{0;\tau}}$ has already been defined as
a function of $u|_{\intco{0;\tau}}$ such that
\ref{th:OptimalityPrinciples:proof:e:3} holds. Then the inequality
$W \le P(W)$ implies that given $u|_{\intco{0;t+1}}$, there is
some $q \in F( \beta(u)(t), u(t) )$ such that
$
W(\beta(u)(t))
\le
2^{-(t+1)} \varepsilon
+
g( \beta(u)(t), q, u(t) )
+
W(q)
$.
Hence, the choice $\beta(u)(t+1) = q$ defines
$\beta(u)|_{\intcc{0;t+1}}$ as a function of $u|_{\intco{0;t+1}}$
such that \ref{th:OptimalityPrinciples:proof:e:3}
holds with $t + 1$ in place of $t$. This proves our claim, and
completes the proof.
\end{proof}

For our representation of the value function as the semi-continuous
limit of value iteration, i.e., of successive applications of the
dynamic programming operator $P$ to the terminal cost function $G$, we
consider the following hypothesis.

\begin{hypothesisA}
\label{h:OptimalityPrinciples:MAX}
$X$ and $U$ are metric spaces,
$F$ is compact-valued, and
$g$, $G$ and $F$ are u.s.c..
\end{hypothesisA}

\begin{corollary}
\label{cor:th:OptimalityPrinciples:MAX}
In the setting of Theorem \ref{th:OptimalityPrinciple}, additionally
assume \ref{h:OptimalityPrinciples:MAX}.
Then $V$ is u.s.c.,
\begin{equation}
\label{cor:th:OptimalityPrinciples:MAX:e:1}
V(p) = \lim_{T \to \infty}P^T(G)(p)
\end{equation}
for all $p \in X$, and $P^{T+1}(G) \le P^T(G)$ for all
$T \in \mathbb{Z}_{+}$.
\end{corollary}

\begin{proof}
Obviously, $0 \le P(W) \le G$ for every
$W \colon X \to \intcc{0, \infty }$, and $P$ is
monotone. This proves the monotonicity claim and shows that
the limit on the right hand
side of \ref{cor:th:OptimalityPrinciples:MAX:e:1}, which we will
denote by $V_{\infty}(p)$, exists in $\intcc{0,\infty}$. In addition,
the inequality $V \le G$ implies $V \le P^T(G)$ for all
$T \in \mathbb{Z}_{+}$, hence $V \le V_{\infty}$.
Next, using Berge's Maximum Theorem \ref{th:BergesMaximumTheorem} and the
fact that the infimum of u.s.c. maps is u.s.c., we see that $P(W)$ is
u.s.c. if $W$ is.
Thus, $P^T(G)$ is u.s.c.~for
every $T \in \mathbb{Z}_{+}$. Then $V_{\infty}$ is u.s.c. as it is the infimum of
u.s.c.~maps.

It remains to show that $V_{\infty} \le P(V_{\infty})$. Then Theorem
\ref{th:OptimalityPrinciple} implies that $V_{\infty} \le V$, and so
$V_{\infty} = V$. Indeed, as $P^T(G)$ is monotonically decreasing
in $T$, we may apply Proposition
\ref{prop:usc_stability_of_monotone_convergence} with
$f_k(q) \defas g(p,q,u) + P^k(G)(q)$ to see that
\[
\lim_{T \to \infty}
\sup_{q}
g(p,q,u) + P^T(G)(q)
=
\sup_{q}
g(p,q,u) + V_{\infty}(q)
\]
for arbitrary $p$ and $u$,
where the suprema are over $q \in F(p,u)$.
As the limit is an infimum, we have
$\lim_{T \to \infty} P(P^T(G))(p) = P( V_{\infty} )(p)$, which
completes the proof.
\end{proof}

We note that while Theorem \ref{th:OptimalityPrinciple} does not
assume any regularity of the problem data, the hypothesis
\ref{h:OptimalityPrinciples:MAX} mandates that e.g.~in the Reach-Avoid
Problem of Example \ref{ex:ReachAvoid} the target set and the
obstacle set is open and closed, respectively. Moreover, if any
one of the assumptions in \ref{h:OptimalityPrinciples:MAX} is dropped,
then the identity \ref{cor:th:OptimalityPrinciples:MAX:e:1} fails to
hold, in general. Also our assumptions of semi-continuity and
compact-valuedness in \ref{h:OptimalityPrinciples:MAX} are automatically satisfied
if the state and input alphabets are finite.

We would like to emphasize that while fixed-point characterizations
and value iteration methods are well known in the field of Dynamic
Programming,
e.g.~\cite{BertsekasShreve96,Bertsekas15,Bertsekas13,DubinsSavage14,MaitraSudderth96},
the available results do not apply in our
setting. Specifically, the theory in \cite{BertsekasShreve96} requires
that cost functionals are represented as limits of finite horizon
costs, which is impossible for the functional in
\ref{e:cfunctional}. The hypotheses in \cite{Bertsekas15} imply that
the dynamic programming operator has a unique fixed-point, and so are
not satisfied by e.g.~the Reach-Avoid Problem of Example
\ref{ex:ReachAvoid} whenever the transition function $F$ is
single-valued and there exists a state that cannot be forced into the
target set.
Similarly, for the unconstrained Minimum Time Problem of
Example \ref{ex:MinTime}, the hypotheses in \cite{Bertsekas13} imply
that the entry time is finite for every state
\cite[Sect.~3.2.1]{Bertsekas13}, or alternatively, that there exists a
uniform bound on all finite entry times
\cite[Sect.~3.2.2]{Bertsekas13}. These assumptions are typically not
satisfied if the state alphabet of the plant is infinite, and are not
imposed in the present paper.
Results on stochastic games,
e.g.~\cite{DubinsSavage14,MaitraSudderth96}, can be directly
interpreted in our setting only if the transition function of the
plant is single-valued. In addition, running costs are typically
assumed to vanish and terminal costs are required to be
real-valued. Moreover, the class of controllers is also restricted,
which can be seen from the result
\cite[Ch.~2.9, Th.~1]{DubinsSavage14} which does not hold in our
setting:
If the state alphabet of the plant is finite and the controller
eventually stops every solution of the closed-loop, then the stopping
times are uniformly bounded.

\section{Comparison of Closed-Loop Performances}
\label{s:ComparisonRefinementAbstraction}

In this section, we introduce
\concept{valuated alternating simulation relations} and
\concept{valuated feedback refinement relations} between optimal
control problems, which are novel, quantitative variants of known
qualitative system relations. As we shall show, the former
concept allows for the efficient comparison of value functions of
related optimal control problems, while the latter guarantees that the
concrete closed-loop value function is upper-bounded, in a
well-defined sense, by the abstract closed-loop value function.
These results will be needed in the proofs of our
main results in Section \ref{s:convergence}.

\subsection{Comparison of value functions}

\begin{definition}
\label{def:valuatedAlternatingSimulationRelation}
Consider optimal control problems
\begin{equation}
\label{e:two:ocp}
\Pi_i = (X_i, U_i, F_i, G_i, g_i ),
\end{equation}
and denote the dynamic programming operator associated with the
problem $\Pi_i$ by $P_i$, $i \in \{1,2\}$.
The relation $Q \colon X_1 \rightrightarrows X_2$ is a
\concept{valuated alternating simulation relation}
from $\Pi_1$ to $\Pi_2$, denoted by
$\Pi_1 \preccurlyeq_Q^\circ \Pi_2$,
if the following conditions hold for all $(p_1,p_2) \in Q$ and all
$u_2 \in U_2$:
\begin{enumerate}
\item
\label{def:valuatedAlternatingSimulationRelation:terminalCost}
$G_1(p_1) \le G_2(p_2)$;
\item
\label{def:valuatedAlternatingSimulationRelation:dyn}
if $G_1(p_1) > 0$ and
the maps $g_2( p_2, \cdot, u_2)$ and $(P_1(0)) \circ Q^{-1}$ are
bounded on the set $F_2(p_2,u_2)$,
where $0$ denotes the zero function on $X_1$,
then for all $\varepsilon > 0$ we have:
\begin{equation}
\label{e:def:valuatedAlternatingSimulationRelation:dyn}
\begin{split}
\exists_{u_1 \in U_1}
\forall_{q_1 \in F_1(p_1,u_1)}
\exists_{q_2 \in F_2(p_2,u_2) \cap Q(q_1)}
\hspace{3em}
\\
g_1(p_1,q_1,u_1) \le \varepsilon + g_2(p_2,q_2,u_2)
.
\end{split}
\end{equation}
\end{enumerate}
\end{definition}

The notion of valuated alternating simulation relation is
related to its well-known qualitative variant in \cite{Tabuada09}
as well as to the quantitative variants employed in
\cite{MazoTabuada10b,deRooMazo13,LeongPrabhakar16}. The concepts in
\cite{Tabuada09,MazoTabuada10b,deRooMazo13,LeongPrabhakar16} require
that the first line of condition
\ref{e:def:valuatedAlternatingSimulationRelation:dyn} holds for all
$(p_1, p_2) \in Q$ and all $u_2 \in U_2$, which implies, roughly
speaking, behavioral inclusion between the two dynamical systems
underlying the optimal control problems $\Pi_1$ and $\Pi_2$.
It is the weaker conditions imposed in Definition
\ref{def:valuatedAlternatingSimulationRelation} that facilitate the
application of valuated alternating simulation relations in our
convergence proof in Section \ref{ss:convergence:convergence}, where behavioral
inclusion cannot be presumed. Comparison of the value functions
associated with two related optimal control problems is still possible
using our fixed-point characterization in Theorem
\ref{th:OptimalityPrinciple}:

\begin{theorem}
\label{th:vASR_Bound_On_V}
Let $\Pi_1$ and $\Pi_2$ be two optimal control problems
with value functions $V_1$ and $V_2$, respectively.
If $\Pi_1 \preccurlyeq_Q^\circ \Pi_2$, then $V_1(p_1) \le V_2(p_2)$
for every $(p_1, p_2) \in Q$.
\end{theorem}

\begin{proof}
Suppose that $\Pi_i$ is of the form~\ref{e:two:ocp}
and let $P_i$ be the associated dynamic programming operator,
$i \in \{ 1, 2 \}$.
We claim that
$P_1(V_1)(p_1) \le P_2(W)(p_2)$ for all $(p_1, p_2) \in Q$, where
the function $W \colon X_2 \to \intcc{0,\infty}$ is defined by
\begin{equation}
\label{e:th:vASR_Bound_On_V:proof:1}
W(p_2)
= \sup \Menge{ V_1(p_1) }{ (p_1, p_2) \in Q }.
\end{equation}
Then, by applying Theorem \ref{th:OptimalityPrinciple} twice, we obtain
$W \le P_2(W)$, and in turn, $W \le V_2$, which proves the assertion.

Let $(p_1, p_2) \in Q$. Our claim is obvious if $G_1(p_1) = 0$, so we
may assume throughout that $G_1(p_1) > 0$. Moreover, from
Definition
\ref{def:valuatedAlternatingSimulationRelation}\ref{def:valuatedAlternatingSimulationRelation:terminalCost},
we see that it suffices to prove that
\begin{multline}
\label{e:th:vASR_Bound_On_V:proof:2}
\inf_{u_1 \in U_{1}}
\sup_{q_1 \in F_1(p_1,u_1)}
g_1(p_1,q_1,u_1) + V_1(q_1)
\ifCLASSOPTIONtwocolumn\\\fi
\le
\sup_{q_2 \in F_2(p_2,u_2)}
g_2(p_2,q_2,u_2) + W(q_2)
\end{multline}
holds for all $u_2 \in U_2$.

Let $u_2 \in U_2$, denote the value of the right hand side of
\ref{e:th:vASR_Bound_On_V:proof:2} by $R$, and suppose that
$R < \infty$. Then the map $g_2( p_2, \cdot, u_2)$ is bounded on the
set $F_2(p_2,u_2)$. The same holds for the map $(P_1(0)) \circ Q^{-1}$
as $V_1 = P_1(V_1) \ge P_1(0)$. Thus, we may assume that
\ref{e:def:valuatedAlternatingSimulationRelation:dyn} holds.
Moreover, the estimate \ref{e:th:vASR_Bound_On_V:proof:2} holds
if for all $\varepsilon > 0$ there exists $u_1 \in U_1$ such that
$
\sup_{q_1 \in F_1(p_1,u_1)}
g_1(p_1,q_1,u_1) + V_1(q_1)
\le
\varepsilon
+
R
$.
This, in turn, is guaranteed if for all $q_1 \in F_1(p_1,u_1)$ there
exists $q_2 \in F_2(p_2,u_2)$ satisfying
$
g_1(p_1,q_1,u_1) + V_1(q_1)
\le
\varepsilon + g_2(p_2,q_2,u_2) + W(q_2)
$, and so an application of
\ref{e:def:valuatedAlternatingSimulationRelation:dyn} and
\ref{e:th:vASR_Bound_On_V:proof:1} completes the proof.
\end{proof}

\subsection{Controller refinement and comparison of closed-loop value
functions}

We have just seen that the existence of a valuated alternating
simulation relation between optimal control problems implies a
comparison between the respective value functions. We now proceed to
introduce the stronger notion of valuated feedback refinement relation
to additionally facilitate the refinement of solutions of one of
the two problems, to the other problem, which is needed in
the proof of one of our main results in Section \ref{s:convergence}.

\begin{definition}
\label{def:FeedbackRefinementRelation}
Consider two optimal control problems $\Pi_1$ and $\Pi_2$ of the form
\ref{e:two:ocp}.
The relation $Q \colon X_1 \rightrightarrows X_2$ is a
\concept{valuated feedback refinement relation}
from $\Pi_1$ to $\Pi_2$, denoted
$\Pi_1 \preccurlyeq_Q \Pi_2$,
if $Q$ is strict and the following conditions
hold for all $(p_1,p_2),(q_1, q_2) \in Q$ and all $u \in U_2$:
\begin{enumerate}
\item
\label{def:FeedbackRefinementRelation:in}
$U_2 \subseteq U_1$;
\item
\label{def:FeedbackRefinementRelation:terminalCost}
$G_1(p_1) \le G_2(p_2)$;
\item
\label{def:FeedbackRefinementRelation:runningCost}
$g_1(p_1,q_1,u) \le g_2(p_2,q_2,u)$;
\item
\label{def:FeedbackRefinementRelation:dyn}
$Q(F_1(p_1,u)) \subseteq F_2(p_2, u)$.
\qedhere
\end{enumerate}
\end{definition}

We first note that every valuated feedback refinement
relation is also a valuated alternating simulation relation. We state
this simple fact as a formal result for later reference:

\begin{proposition}
\label{prop:vFRR_is_also_sASR}
$\Pi_1 \preccurlyeq_Q \Pi_2$
implies
$\Pi_1 \preccurlyeq_Q^\circ \Pi_2$.
\end{proposition}

Apart from conditions
\ref{def:FeedbackRefinementRelation:terminalCost} and
\ref{def:FeedbackRefinementRelation:runningCost} in
Definition~\ref{def:FeedbackRefinementRelation},
and in the special case of strict transition functions considered in
the present paper, the notion of valuated feedback refinement relation
coincides with its qualitative variant introduced
in \cite{i14sym}.
Hence, we can take advantage of the controller refinement scheme
presented in \cite{i14sym}. That is, we refine any abstract controller
by serially connecting it with a valuated feedback refinement relation
used as an interface; see \ref{fig:closedLoopVFRR}. We therefore need
to formalize the concept of serial composition:
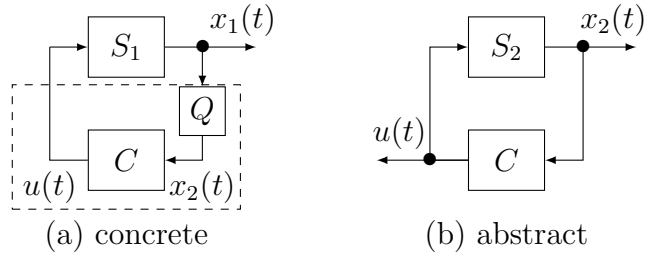
\begin{figure}[hbt]
\begin{centering}
\begin{tikzpicture}[node distance=1.8cm, >=latex]
  \def\dx{0.15}

  \tikzset{
    block/.style = {draw,
                    rectangle,
                    minimum height = .8cm,
                    minimum width = 1cm},
    interface/.style = {draw,
                        rectangle,
                        minimum height = .5cm,
                        minimum width = .5cm}}
  \draw node at (5,0) [block] (F)  {$S_2$};
  \draw node at (5,-1.5) [block] (C)  {$C$};

  \draw node at (5,-2.5) {(b) abstract};

  \draw[->] (C.west) -- node[above,near end] {$u(t)$} ++(-1.2,0);
  \draw[->] (C.west) -| ($(F.west)-(.5,0)$) -- ($(F.west)-(0,0)$);

  \draw[->] ($(F.east)+(.5,0)$) |- (C.east);
  \draw[->] (F.east)  -- node[above,near end] {$x_2(t)$} ++(1.2,0);

  \draw[->] node at ($(F.east)+(.5,0)$) {\textbullet};
  \draw[->] node at ($(C.west)+(-.5,0)$) {\textbullet};

  \draw node at (0,0) [block] (F1)  {$S_1$};
  \draw node at (0,-1.5) [block] (C1)  {$C$};
  
  \draw node at (0,-2.5) {(a) concrete};

  \draw node at (1,-.85) [interface] (Q) {$Q$};


  \draw[dashed] (-1.5,-0.5) rectangle (1.5,-2.15);

  \draw (C1.west) -| node[below] {$u(t)$} ($(F1.west)+(-0.5,0)$);
  
  \draw[->] ($(F1.west)+(-0.5,0)$) |- (F1.west);



  \draw[->] (F1.east) -| (Q.north);
  \draw[->] (Q.south) |- node[below] {$x_2(t)$} (C1.east);
  \draw[->] ($(F1.east)+(0.5,0)$)  -- node[above,near end] {$x_1(t)$} ++(.7,0);

  \draw[->] node at ($(F1.east)+(.5,0)$) {\textbullet};

\end{tikzpicture}
\end{centering}
\caption{\label{fig:closedLoopVFRR}
\label{review:item6:text:b}
Using a valuated feedback refinement relation $Q$ from $S_1$ to $S_2$,
an abstract controller $C$ is refined into the serial
composition of $Q$ and $C$.
}
\end{figure}

\begin{definition}
\label{def:serial}
Let $C$ be a controller of the form \ref{e:def:Controller},
$X'$
be a non-empty set and
$Q \colon X'\rightrightarrows \widetilde{X}$
be a strict map. The \concept{serial composition of $Q$ and $C$},
denoted $C \circ Q$, is the
controller
$( Z, Z_{0}, X', \widetilde{U}, H' )$
with
$H'(z,x')=H(z,Q(x'))$.
\end{definition}

As demonstrated in \cite{i14sym} the proposed controller refinement
scheme implies a comparison between closed-loop behaviors. Here we
extend that result to guarantee a comparison between closed-loop
value functions:

\begin{theorem}
\label{th:ControllerRefinement}
Let $\Pi_1$ and $\Pi_2$ be optimal control problems of the
form~\ref{e:two:ocp}, and suppose that $\Pi_1 \preccurlyeq_Q \Pi_2$
and $C \in \mathcal{F}(X_2,U_2)$.
Then $C \circ Q \in \mathcal{F}(X_1,U_1)$ and we have
\begin{equation}
\label{e:th:ControllerRefinement}
\forall_{p_1 \in X_1}
L_1(p_1)
\le
\sup L_2( Q(p_1) ),
\end{equation}
where $L_1$  and $L_2$ are the closed-loop value functions of $\Pi_1$  and $\Pi_2$
associated with $C\circ Q$  and $C$, respectively.
\end{theorem}

\begin{proof}
The fact that $C \circ Q \in \mathcal{F}(X_1,U_1)$ is obvious.
Denote the cost functional associated with $\Pi_i$ by $J_i$,
set $S_i = (X_i,U_i,F_i)$, $i \in \{1,2\}$, and let
$(u,v,x_1) \in \mathcal{B}( (C \circ Q) \times S_1 )$.
We claim that there exists a signal
$x_2 \colon \mathbb{Z}_{+} \to X_2$ satisfying
$x_2(0) \in Q( x_1(0) )$,
$(u,v,x_2) \in \mathcal{B}( C \times S_2 )$, and
$J_1(u,v,x_1) \le J_2(u,v,x_2)$. This implies
\ref{e:th:ControllerRefinement} and completes our proof.

To prove our claim, we first note that there exists a signal
$z$ defined on $\mathbb{Z}_{+}$ such that $(u,v,z,x_1)$ is a solution
of $C \circ Q$ and $(u,x_1)$ is a solution of $S_1$.
By the former fact and Definitions \ref{def:System} and
\ref{def:serial}, there exists a signal
$x_2 \colon \mathbb{Z}_{+} \to X_2$ such that
$(u,v,z,x_2)$ is a solution of $C$ and
$( x_1(t), x_2(t) ) \in Q$ for all $t \in \mathbb{Z}_{+}$.
Using
\ref{def:FeedbackRefinementRelation:dyn} in Definition
\ref{def:FeedbackRefinementRelation} we obtain
$x_2(t+1)
\in
Q(x_1(t+1))
\subseteq
Q( F_1( x_1(t), u(t) ) )
\subseteq
F_2( x_2(t), u(t) )$ for all $t$. Hence, $(u,x_2)$ is a solution of
$S_2$, and so $(u,v,x_2) \in \mathcal{B}(C \times S_2)$.
We obviously have $J_1(u,v,x_1) \le J_2(u,v,x_2)$ if $v = 0$, and if
$v \neq 0$ the same estimate follows from
\ref{def:FeedbackRefinementRelation:terminalCost} and
\ref{def:FeedbackRefinementRelation:runningCost} in
Definition~\ref{def:FeedbackRefinementRelation}.
\end{proof}

For easier reference in later sections, we reformulate Theorem
\ref{th:ControllerRefinement} in terms of pointwise upper performance
bounds:

\begin{definition}
\label{def:PointwiseUpperBound}
Let $Q \colon X_1 \rightrightarrows X_2$ be strict and
let $f \colon X_2 \to \intcc{0,\infty}$.
Then the function $\hat f^{(Q)} \colon X_1 \to \intcc{0,\infty}$
defined by
\[
\hat f^{(Q)} (x)
=
\sup
f( Q( x ) )
\]
is called
\concept{pointwise upper bound of $f$ associated with $Q$}.
\end{definition}

\begin{corollary}
\label{cor:th:ControllerRefinement}
Under the hypotheses and in the notation of Theorem
\ref{th:ControllerRefinement} we have
$L_1 \le \hat L_2^{(Q)}$.
\end{corollary}

\section{Main results}
\label{s:convergence}

In this section, we
introduce a notion of abstraction of optimal control problems which
comes with a non-negative conservatism parameter. We will then
show that the concrete value function can be approximated arbitrarily
closely using value functions of sufficiently precise abstractions.
Moreover, we shall show that if abstract controllers can be chosen to
be optimal, the performance of the closed-loop in
\ref{fig:closedLoopVFRR} converges to the concrete value function as
well.
The latter result implies a kind of completeness property of
controller synthesis based on abstractions of conservatism introduced in
this paper, an aspect to be discussed at the end of the section.

\subsection{Abstractions and their conservatism}
\label{ss:Abstractions}

To begin with, we first introduce abstractions devoid of any
notion of conservatism. In doing so, we focus on a case
where the abstract state space is a cover of the concrete state space,
which has turned out to be canonical in the qualitative setting
\cite[Sec.~VII]{i14sym}.
Here, a \concept{cover} of a set $X$ is a set of subsets of $X$ whose
union equals $X$.

\begin{definition}
\label{def:Abstraction}
Let $\Pi_1$ and $\Pi_2$ be optimal control problems  of the
form~\ref{e:two:ocp}, where $X_2$ is a
cover of $X_1$ by non-empty subsets.
Then $\Pi_2$ is an \concept{abstraction} of $\Pi_1$ if
$\Pi_1 \preccurlyeq_{\in} \Pi_2$, where
$\mathop{\in} \colon X_1 \rightrightarrows X_2$ denotes the membership
relation.
\end{definition}

\noindent
For later reference, we explicitly state our requirements on
abstractions.

\begin{proposition}
\label{prop:Abstraction}
Let $\Pi_1$ and $\Pi_2$ be optimal control problems  of the
form~\ref{e:two:ocp}, where $X_2$ is a
cover of $X_1$ by non-empty subsets.
Then $\Pi_1 \preccurlyeq_{\in} \Pi_2$ iff the following conditions
hold whenever
$p \in \Omega \in X_2$, $p' \in \Omega' \in X_2$ and $u \in U_2$:
\begin{enumerate}
\item
\label{prop:Abstraction:in}
$U_2\subseteq U_1$;
\item
\label{prop:Abstraction:terminalCost}
$
G_1(p)
\le
G_2(\Omega)
$;
\item
\label{prop:Abstraction:runningCost}
$
g_1(p,p',u)
\le
g_2(\Omega,\Omega',u)
$;
\item
\label{prop:Abstraction:iii}
$
\Omega' \cap F_1(\Omega,u) \neq \emptyset
\implies
\Omega'\in F_2(\Omega,u)
$.
\qedhere
\label{prop:Abstraction:last}
\end{enumerate}
\end{proposition}

\begin{proof}
Obviously, the relation $\mathop{\in}$ is strict as $X_2$ is
a cover of $X_1$, and if $Q = \mathop{\in}$, then
the conditions \ref{prop:Abstraction:in} through
\ref{prop:Abstraction:runningCost}
are equivalent to the respective conditions
in Definition~\ref{def:FeedbackRefinementRelation}.
The equivalence of condition \ref{prop:Abstraction:iii}
to the condition
\ref{def:FeedbackRefinementRelation:dyn} in
Definition~\ref{def:FeedbackRefinementRelation} is obtained by an
application of \cite[Prop.~VII.1]{i14sym} to the systems
$S_i=(X_i,X_i,U_i,U_i,X_i,F_i,\id)$, $i\in\{1,2\}$. 
\end{proof}

As we can see, even rather conservative approximations of the concrete
optimal control problem may qualify as abstractions.
We aim at resolving that issue
by introducing a
suitable notion of \concept{conservatism}.
To this end, we first need to introduce some
additional notation.
For any metric space $(X,d)$ we define
\begin{align*}
d(x,N)
&=
\inf
\Menge{d(x,y)}{y \in N},\\
d(M,N)
&=
\inf
\Menge{d(x,y)}{x \in M, y \in N}
\end{align*}
for all $x \in X$ and all nonempty subsets
$M, N \subseteq X$.
We use $\oBall(c,r)$ and $\cBall(c,r)$ to denote the open,
respectively, closed ball with center $c\in X$ and radius
$r > 0$, and we adopt the convention that $\cBall(c,0) = \{ c \}$.
We denote the diameter of a subset $M \subseteq X$ by $\diam(M)$.
See
\cite{HuPapageorgiou97.i}.

\begin{definition}
\label{def:AbstractionOfPrecision}
Let $\Pi_2$ be an abstraction of $\Pi_1$ and suppose that
$\Pi_1$ and $\Pi_2$ are of the form~\ref{e:two:ocp},
that $U_1$ and $X_1$ are metric spaces, and that
the elements of $X_2$ are closed subsets of $X_1$.
\\
Then $\Pi_2$ is an \concept{abstraction of
conservatism
$\infty$} of
$\Pi_1$.
Moreover, 
$\Pi_2$ is an \concept{abstraction of
conservatism
$\rho \in \mathbb{R}_{+}$} of
$\Pi_1$ if the following conditions hold
for all $\Omega, \Omega' \in X_2$ and all $u \in U_2$:
\begin{enumerate}
\item
\label{def:AbstractionOfPrecision:in}
$U_1 = \cBall(U_2,\rho)$;
\item
\label{def:AbstractionOfPrecision:terminalCost}
$
G_2(\Omega)
\le
\rho
+
\sup G_1(\Omega)
$;
\item
\label{def:AbstractionOfPrecision:runningCost}
$
g_2(\Omega,\Omega',u)
\le
\rho
+
\sup g_1(\Omega, \Omega', u)
$.
\end{enumerate}
If $\Omega$ satisfies the condition
\begin{equation}
\label{e:def:AbstractionOfPrecision}
G_1(\Omega)
\cup
g_1(\Omega,X_1,U_1)
\not=
\{ \infty \}
,
\end{equation}
then we additionally require the following:
\begin{enumerate}[resume]
\item
\label{def:AbstractionOfPrecision:iii}
$F_2(\Omega,u)
\subseteq
\Menge{\Omega'' \in X_2}%
{ d(\Omega'', F_1(\Omega, u ) ) \le \rho}
$,
where $d$ denotes the metric on $X_1$;
\item
\label{def:AbstractionOfPrecision:diam}
$\diam( \Omega ) \le \rho$.
\label{def:AbstractionOfPrecision:last}
\qedhere
\end{enumerate}
\end{definition}

As we had announced, Definition \ref{def:AbstractionOfPrecision}
limits the conservatism of abstractions.
Specifically, while the conditions \ref{prop:Abstraction:in} through
\ref{prop:Abstraction:iii} in Proposition
\ref{prop:Abstraction} demand that
$U_1$, $G_2(\Omega)$, $g_2(\Omega,\Omega',u)$ and $F_2(\Omega,u)$
merely over-approximate
$U_2$, $\sup G_1(\Omega)$, $\sup g_1(\Omega,\Omega',u)$ and
$F_1(\Omega,u)$,
respectively, the respective conditions in Definition
\ref{def:AbstractionOfPrecision} mandate that the approximation
error does not exceed the value of the conservatism parameter $\rho$, and
\ref{def:AbstractionOfPrecision:diam} bounds the error by which
abstract states over-approximate concrete states.
The condition \ref{e:def:AbstractionOfPrecision} restricts
the requirements \ref{prop:Abstraction:iii} and
\ref{def:AbstractionOfPrecision:diam} to regions where the concrete
value function is possibly finite.

\subsection{Arbitrarily close approximation of concrete value functions}
\label{ss:convergence:convergence}

We next need to choose a suitable notion of convergence. On
the one hand, pointwise convergence is not powerful enough, e.g. to
imply our completeness results in Section \ref{ss:AlgorithmicSolution:OverallMethod},
and similarly for convergence in Lebesgue spaces
as employed in \cite{GrueneJunge08}.
On the other hand, the stronger concept of uniform convergence would
require that any points of discontinuity
of the concrete value function are
also present, exactly and not only approximately, in the functions to
approximate it, which is not realistic to assume. We here
rely on a concept that lies in between the aforementioned
extremes, and the first main result of our paper shows that the
hypographs of pointwise upper bounds of the abstract value functions locally
approximate the hypograph of the concrete value function.
See \ref{fig:hypoConv}.
\begin{figure}[t]
\begin{center}
\noindent
\hspace*{\fill}
\psfrag{BhypoVe}[][]{$\oBall(\hypo(V),\varepsilon)$}
\psfrag{G}[][]{$N$}
\psfrag{e}[r][r]{$\varepsilon$}
\psfrag{hypoV}[][]{$\hypo(V)$}
\psfrag{V}[][]{$V$}
\psfrag{W}[][]{$W$}
\psfrag{Rn}[r][r]{$X$}
\psfrag{R}[r][r]{$\mathbb{R}$}
\includegraphics[width=\ifCLASSOPTIONonecolumn.35\else.7\fi\linewidth]{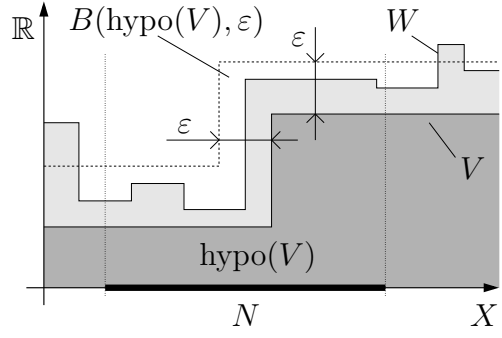}%
\hspace*{\fill}
\end{center}
\caption{\label{fig:hypoConv}%
Approximation of the hypograph of the map
$V \colon X \to \mathbb{R}_{+} \cup \{ \infty \}$ by the hypograph of
$W \ge V$,
on the subset $N \subseteq X$ \cite{Reissig17DPc}.%
}
\end{figure}
The result requires tightening the hypothesis
\ref{h:OptimalityPrinciples:MAX} on the optimal control problem
\ref{e:OCP} as follows:
\begin{hypothesisA}
\label{h:Convergence}
$X$ is a proper metric space,
$U$ is a compact metric space,
$F$ is compact-valued, and
$g$, $G$ and $F$ are u.s.c..
\end{hypothesisA}
Here, a metric space is \concept{proper} if every closed ball
is compact, a requirement satisfied, e.g. by $\mathbb{R}^n$ and
by all of its closed metric subspaces.
Hypothesis \ref{h:Convergence} is extended
to optimal control problems $\Pi_i$ of the form~\ref{e:two:ocp}
in the obvious way.
In the following, we do not mention explicitly the association of
pointwise upper bounds on abstract value functions with the respective
membership relations.

\begin{theorem}
\label{th:UpperBoundAndConvergence}
Let $\Pi$ be the optimal control problem \ref{e:OCP} and let $V$
denote the value function of $\Pi$.
Then the pointwise upper bound of the value function of any
abstraction of $\Pi$ is an upper bound on $V$.
If \ref{e:OCP} additionally satisfies~\ref{h:Convergence}, then
for every $p \in X$ and every $\varepsilon > 0$ there exist a
neighborhood $N \subseteq X$ of $p$ and
some $\rho \in \mathbb{R}_{+} \setminus \{ 0 \}$ such that
\begin{equation}
\label{e:th:UpperBoundAndConvergence}
( N \times \mathbb{R} )
\cap
\hypo W
\subseteq
\oBall( \hypo V, \varepsilon )
\end{equation}
holds whenever $W$ is the pointwise upper bound on the value
function of an abstraction of conservatism $\rho$ of \ref{e:OCP}.
\end{theorem}

To prove the theorem we will introduce an auxiliary optimal control
problem $\Pi_3$ with the following properties.
Firstly, the state space $X_3$ of $\Pi_3$ comprises both a copy of the
concrete state space and (almost the whole of) the state spaces of
all abstractions, of arbitrary conservatism. Secondly, the value function
$V_3$ of $\Pi_3$ restricted to the concrete state space coincides with
the concrete value function $V$.
Thirdly, $V_3$ is an upper bound on any abstract value function, on
the respective subset of $X_3$.
Using the semi-continuity of $V_3$ on the whole of $X_3$, we then
conclude that the abstract value function arbitrarily
closely approximates $V$ whenever the abstract state
space sufficiently closely approximates the concrete one.

In our proof below, the notion of \concept{graph} of a set-valued map
$f \colon X \rightrightarrows Y$ refers to the set
$\Menge{(x,y) \in X \times Y}{y \in f(x)}$, and we also use
the space $K(X)$ of non-empty compact subsets of $X$
endowed with the Hausdorff metric
\cite{RockafellarWets09,HuPapageorgiou97.i}
associated with the metric on $X$,
and its subspaces $K_{\rho}(X)$ defined by
\[
K_{\rho}(X)
=
\Menge{ \Omega \in K(X)}{\diam \Omega \leq \rho}.
\]

\begin{lemma}
\label{lem:AdmissibilityOfAuxiliaryProblem}
Let $\Pi_1$ be an optimal control problem of the form \ref{e:two:ocp} that
satisfies~\ref{h:Convergence},
and denote the metric on $X_1$ by $d$. Let
$\Pi_3=(X_3,U_3,F_3,G_3,g_3)$ be given by $X_3=K(X_1) \times
\mathbb{R}_{+}$, $U_3=U_1$ and
\begin{align*}
  &F_3((\Omega,\rho),u)
=\\
&\quad\qquad\Menge{ \Omega' \in K_{\rho}(X_1) }%
{ d( \Omega', F_1( \Omega, \cBall( u, \rho ) ) ) \le \rho }
\times \{\rho\},\\
&G_3((\Omega, \rho))
=
\rho
+
\sup G_1(\Omega),
\\
&g_3((\Omega,\rho),(\Omega',\rho'),u)
=
\rho
+
\sup g_1( \Omega, \Omega', \cBall(u,\rho) ).
\end{align*}
Then
$\Pi_3$
is an optimal control problem satisfying~\ref{h:Convergence}.
\end{lemma}
\begin{proof}
$\Pi_3$ is clearly an optimal control problem by our
hypotheses, and in particular, $F_3$ is strict. Moreover, $U_3$ is
compact, and $K(X_1)$ is proper as $X_1$ is so.
In addition,
using Proposition \ref{prop:uscCompact} it is easily seen that
the maps $\alpha \colon K(X_1) \rightrightarrows X_1$ and
$\beta \colon U_1 \times \mathbb{R}_{+} \rightrightarrows U_1$
given by $\alpha(\Omega) = \Omega$ and $\beta(u,r) = \cBall(u,r)$
are u.s.c. and compact-valued, or \concept{usco} for short.
Then $G_3$ and $g_3$ are u.s.c. by Theorem
\ref{th:BergesMaximumTheorem}.

To show that $F_3$ is usco, define the map
$
H \colon K(X_1) \times U_1 \times \mathbb{R}_{+}
\rightrightarrows
X_1
$
by
$H(\Omega,u,\rho) = \beta(F_1(\alpha(\Omega),\beta(u,\rho)),\rho)$,
let
$(((\Omega_k,\rho_k),u_k),(\Omega_k',\rho_k))_{k \in \mathbb{N}}$ be a sequence
in the graph of $F_3$, and suppose that the sequences $\Omega$, $\rho$
and $u$ converge to $\Gamma \in K(X_1)$, $r \in \mathbb{R}_{+}$ and
$v \in U_1$, respectively.
Then $\Omega_k' \in K_{\rho_k}(X_1)$ for all $k$, and
since $F_1$ is usco, we also have 
$\Omega_k' \cap H(\Omega_k, u_k, \rho_k) \not= \emptyset$ for all
$k$.
Thus, there exists a sequence $(p_k)_{k \in \mathbb{N}}$
satisfying $p_k \in \Omega_k' \cap H(\Omega_k, u_k, \rho_k)$ for all
$k$, and by Proposition \ref{prop:uscCompact}, a
subsequence of $p$ converges to some $q \in H(\Gamma,v,r)$ since $H$
is usco.
We may assume that the whole sequence converges.
Then the sequence $\Omega'$ is bounded, and so may be assumed to
converge to some $\Gamma' \in K(X_1)$ since $K(X_1)$ is proper.
Additionally, $\Gamma' \in K_r(X_1)$ by the continuity of the map
$\diam$ on $K(X_1)$, and $q \in \Gamma'$. We conclude that
$(\Gamma',r) \in F_3((\Gamma,r),v)$, and so $F_3$ is usco by
Prop.~\ref{prop:uscCompact}.
\end{proof}

\begin{lemma}
\label{lem:EstimatesUsingAuxiliaryProblem}
Under the hypotheses and in the notation of Lemma
\ref{lem:AdmissibilityOfAuxiliaryProblem}, let
$\Pi_2$ be an abstraction of conservatism $\rho \in \mathbb{R}_{+}$
of $\Pi_1$, of the form \ref{e:two:ocp}.
Let $V_i$ denote the value function of $\Pi_i$, $i \in \{ 1, 2, 3 \}$,
and let $X_2' \subseteq X_2$ be the subset of cells $\Omega$ that
satisfy \ref{e:def:AbstractionOfPrecision}. Then the following holds:
\begin{enumerate}
\item
\label{lem:EstimatesUsingAuxiliaryProblem:i}
$V_1(p)= V_3(\{ p \}, 0)$ for all $p\in X_1$;
\item
\label{lem:EstimatesUsingAuxiliaryProblem:ii}
$V_2( \Omega ) \le V_3( \Omega, \rho)$
for all $\Omega \in X_2'$;
\item
\label{lem:EstimatesUsingAuxiliaryProblem:iii}
$V_2( \Omega ) \le V_3( \{ p \}, \rho)$
whenever $p \in \Omega \in X_2 \setminus X_2'$.
\qedhere
\end{enumerate}
\end{lemma}

\begin{proof}
We claim that
$\Pi_1 \preccurlyeq^{\circ}_{Q} \Pi_3 \preccurlyeq^{\circ}_{Q^{-1}} \Pi_1$ 
holds for the single-valued map
$Q \colon X_1 \rightrightarrows X_3$ given by
$Q(p) = ( \{ p \}, 0 )$.
Indeed, let $p\in X_1$ and $u \in U_3$. Then
$G_3( Q(p) ) = G_1(p)$ and
$g_3( Q(p), Q(q), u) = g_1(p,q,u)$ for all $q \in X_1$.
Moreover, $Q( F_1(p,u) ) = F_3( (\{ p \},0), u )$ as $F_1$ is
compact-valued. Thus, both conditions in Definition
\ref{def:valuatedAlternatingSimulationRelation} are met with $\Pi_3$
in place of $\Pi_2$, and they are also met with $\Pi_3$ and $\Pi_1$ in
place of $\Pi_1$ and $\Pi_2$, respectively. This proves our claim, and
\ref{lem:EstimatesUsingAuxiliaryProblem:i} follows from Theorem
\ref{th:vASR_Bound_On_V}.

To prove \ref{lem:EstimatesUsingAuxiliaryProblem:ii} and
\ref{lem:EstimatesUsingAuxiliaryProblem:iii} we shall show that
$\Pi_2 \preccurlyeq^{\circ}_{Q} \Pi_3$ holds for the relation
$Q \colon X_2 \rightrightarrows X_3$ given by
$Q( \Omega ) = \{ ( \Omega, \rho ) \}$ if $\Omega \in X_2'$, and by
$
Q( \Omega )
=
\Menge{ ( \{ p \}, \rho )}{p \in \Omega}
$,
otherwise.

Let $( \Omega, ( \Omega', \rho ) ) \in Q$ and
$u_3 \in U_3$. Then $\Omega' \subseteq \Omega$, and
additionally $( \Omega', \rho ) \in X_3$ as required since
$X_2' \subseteq K_{\rho}(X_1)$.
Moreover, the estimate $G_2( \Omega ) \le G_3( \Omega', \rho )$ is
immediate from Definition \ref{def:AbstractionOfPrecision}
if $\Omega \in X_2'$. It also holds if
$\Omega \in X_2 \setminus X_2'$, for then
\ref{e:def:AbstractionOfPrecision} is violated, which implies
$G_3(\Omega',\rho) = \infty$. Hence, the first requirement in
Definition \ref{def:valuatedAlternatingSimulationRelation} is met with
$\Pi_2$ and $\Pi_3$ in place of $\Pi_1$ and $\Pi_2$, respectively.

In our proof of the second requirement we may assume that the map
$g_3((\Omega',\rho),\cdot,u_3)$ is bounded on the set
$F_3((\Omega',\rho),u_3)$. Then
$g_1(\Omega', X_1, u_3) \not= \{ \infty \}$ by the definition of
$g_3$, and so $\Omega = \Omega' \in X_2'$. We next pick any
$u_2 \in \cBall(u_3,\rho) \cap U_2$, which is possible by
condition \ref{def:AbstractionOfPrecision:in} in Definition
\ref{def:AbstractionOfPrecision}, and any
$\Omega'' \in F_2(\Omega,u_2)$. Then the
condition \ref{def:AbstractionOfPrecision:iii} in
Def.~\ref{def:AbstractionOfPrecision} shows that
\begin{equation}
\label{e:th:UpperBoundAndConvergence:proof:1}
d( \Omega'', F_1( \Omega, \cBall(u_3,\rho) ) )
\le
\rho.
\end{equation}

If $\Omega'' \in X_2'$, then
$(\Omega'', \rho) \in F_3((\Omega,\rho),u_3) \cap Q( \Omega'' )$.
Moreover, the condition
\ref{def:AbstractionOfPrecision:runningCost} in Definition
\ref{def:AbstractionOfPrecision} with $\Omega''$ and $u_2$ in place of
$\Omega'$ and $u$, respectively, shows that
$
g_2(\Omega,\Omega'',u_2)
\le
g_3((\Omega,\rho),(\Omega'',\rho),u_3)
$, and we are done.
If, on the other hand, $\Omega'' \not\in X_2'$, then
$G_1( \Omega'' ) \cup g_1( \Omega'', X_1, U_1 ) = \{ \infty \}$, and
hence, $G_2( \Omega'' ) = \infty$ and 
$g_2( \Omega'', X_2, U_2 ) = \{ \infty \}$ by Proposition
\ref{prop:Abstraction}. This shows that
$(P_2(0))( \Omega'' ) = \infty$. Moreover,
$(\{ q \}, \rho) \in F_3((\Omega,\rho),u_3)$
for some $q \in \Omega''$
by \ref{e:th:UpperBoundAndConvergence:proof:1}
and a compactness argument.
Since  additionally
$(\{ q \}, \rho) \in Q (\Omega'')$ it follows that the map
$(P_2(0)) \circ Q^{-1}$ is not bounded on the set
$F_3((\Omega,\rho),u_3)$, which completes our proof.
\end{proof}

\begin{proof}[\proofname{} of Theorem \ref{th:UpperBoundAndConvergence}]
The first claim of the theorem directly follows from
Def.~\ref{def:Abstraction}, Prop.~\ref{prop:vFRR_is_also_sASR}, and
Th.~\ref{th:vASR_Bound_On_V}.
To prove the second claim, let
$\varepsilon > 0$, $p \in X$ and $\rho > 0$,
let $\Pi_i$, $V_i$ and $X_2'$ be as in Lemmas
\ref{lem:AdmissibilityOfAuxiliaryProblem} and
\ref{lem:EstimatesUsingAuxiliaryProblem}, $i \in \{ 1,2,3 \}$,
let $N = \oBall(p,\varepsilon) \subseteq X_1$,
and let $W$ be the pointwise upper bound of $V_2$.
If \ref{e:th:UpperBoundAndConvergence} does not hold with
$V_1$ in place of $V$, then there exists $x \in N$ satisfying
$V_1(p) + \varepsilon / 2 < W(x)$. Then
$V_1(p) + \varepsilon / 2 < V_2( \Omega' )$ for some
$\Omega' \in X_2$ containing $x$, by the definition of $W$, and
$V_2( \Omega' ) \le V_3( \Omega, \rho )$
for some $\Omega \in K_{\rho}( X_1 )$ containing $x$, by
Lemma \ref{lem:EstimatesUsingAuxiliaryProblem}; specifically,
$\Omega = \Omega'$ if $\Omega' \in X_2'$, and
$\Omega = \{ x \}$, otherwise.

We conclude that, if the second claim of the theorem does not hold,
then there exist $\varepsilon > 0$, $p \in X$ and a sequence
$(\Omega_k)_{k \in \mathbb{N}}$ in $K(X_1)$ converging to $\{ p \}$
such that $V_1(p) + \varepsilon / 2 < V_3(\Omega_k,1/k)$ for all
$k \in \mathbb{N}$.
On the other hand, $V_3$ is u.s.c. by Lemma
\ref{lem:AdmissibilityOfAuxiliaryProblem} and Corollary
\ref{cor:th:OptimalityPrinciples:MAX}, and this together with
Lemma
\ref{lem:EstimatesUsingAuxiliaryProblem}\ref{lem:EstimatesUsingAuxiliaryProblem:i}
shows that
$\limsup_{k \to \infty} V_3(\Omega_k,1/k) \le V_1(p)$, which is a
contradiction.
\end{proof}

\subsection{Convergence of the closed-loop performance to the concrete value function}
\label{ss:AlgorithmicSolution:OverallMethod}

Finally, we will demonstrate that the performance of the concrete
closed-loop in \ref{fig:closedLoopVFRR} converges to the concrete
value function, in which we use the following notion of convergence;
see
\cite{RockafellarWets09,HuPapageorgiou97.i}
and Proposition \ref{prop:hypoLimit:i13absoc} in the Appendix.

\begin{definition}
\label{def:hypoLimit:i13absoc}
Let the map $V \colon X \to \mathbb{R}_{+} \cup \{ \infty \}$ be
u.s.c. on the metric space $X$, and let
$L_i \colon X \to \mathbb{R}_{+} \cup \{ \infty \}$ satisfy
$L_i \ge V$, for all $i \in \mathbb{N}$.
Then the sequence $(L_i)_{i \in \mathbb{N}}$ \concept{hypo-converges}
to $V$, denoted $V = \hypolim_{i \to \infty} L_i$, if the following
condition holds.
For every $p \in X$ and every $\varepsilon > 0$ there exist a
neighborhood $N \subseteq X$ of $p$ such that the inclusion
\begin{equation}
\label{e:def:hypoLimit:i13absoc}
( N \times \mathbb{R} )
\cap
\hypo L_i
\subseteq
\oBall( \hypo V, \varepsilon )
\end{equation}
holds for all sufficiently large $i \in \mathbb{N}$.
\end{definition}

In addition to hypothesis \ref{h:Convergence}, throughout the rest of
this section we shall assume the following.

\begin{hypothesisA}
\label{h:ss:AlgorithmicSolution:OverallMethod}
\begin{enumerate}
\item
For every $i \in \mathbb{N}$,
$\Pi_i$ is an abstraction of conservatism
$\rho_i \in \mathbb{R}_{+} \cup \{ \infty \}$
of \ref{e:OCP}, of
the form \ref{e:two:ocp}, $C_i$ is an optimal controller for
$\Pi_i$, and $L_i$ is the closed-loop value function
of \ref{e:OCP} associated with $C_i \circ \mathop{\in}$, where
$\mathop{\in} \colon X \rightrightarrows X_i$ is the membership
relation and $\lim_{i \to \infty} \rho_i = 0$.
\item
$V$ is the value function of \ref{e:OCP}.
\end{enumerate}
\end{hypothesisA}

Here, $C_i \in \mathcal{F}(X_i,U_i)$ is an
\concept{optimal controller} for $\Pi_i$ if the value function of
$\Pi_i$ coincides with the closed-loop value function of $\Pi_i$
associated with $C_i$, i.e., if $C_i$ realizes the achievable
performance of the abstract closed-loop.
As detailed in Section \ref{s:AlgorithmicSolution}, optimal abstract
controllers exist whenever abstractions are finite, and finite,
arbitrarily precise abstractions can actually be computed
in the case of sampled-data control system dynamics.

We are now ready to present our second main result.

\begin{theorem}
\label{th:HypoConvergence}
Assume \ref{h:Convergence}, \ref{h:ss:AlgorithmicSolution:OverallMethod}.
Then $\hypolim\limits_{i \to \infty} L_i = V$.
\end{theorem}

\begin{proof}
Obviously, $L_i \ge V$ for all $i$, $X$
is a metric space, and $V$ is u.s.c. by Corollary
\ref{cor:th:OptimalityPrinciples:MAX}. Let $W_i$ be the value function
of $\Pi_i$, and let $p \in X$ and $\varepsilon > 0$.
By Theorem \ref{th:UpperBoundAndConvergence} there exists a
neighborhood $N \subseteq X$ of $p$ such that
$
( N \times \mathbb{R} )
\cap
\hypo \hat W_i^{(\in)}
\subseteq
\oBall( \hypo V, \varepsilon )
$
holds for all sufficiently large $i \in \mathbb{N}$.
Then, since $L_i \le \hat W_i^{(\in)}$ for all $i$ by
Corollary
\ref{cor:th:ControllerRefinement}, the requirement in Definition
\ref{def:hypoLimit:i13absoc} is satisfied.
\end{proof}

The theorem implies that the concrete value function $V$
is uniformly approximated on compact sets by the actual closed-loop
performances $L_i$. Specifically, for every $\varepsilon > 0$ and
every compact subset $N \subseteq X$ the inclusion
\ref{e:def:hypoLimit:i13absoc} holds for all sufficiently large
$i \in \mathbb{N}$. See also \ref{fig:hypoConv}.
Theorem \ref{th:HypoConvergence} also implies
pointwise convergence, and it even implies uniform convergence on any
set on which such a strong convergence property can possibly be expected:

\begin{corollary}
\label{cor:th:HypoConvergence}
Assume \ref{h:Convergence} and \ref{h:ss:AlgorithmicSolution:OverallMethod}.
Then we have
\begin{equation}
\label{e:PointwisConvergence}
V(p) = \lim_{i \to \infty} L_i(p)
\;\;\;\text{for all $p \in X$},
\end{equation}
and the following holds for every compact subset
$N \subseteq X$.
\begin{enumerate}
\item
\label{cor:th:HypoConvergence:Bound}
For every $\varepsilon > 0$ and all sufficiently large
$i \in \mathbb{N}$ we have
$\sup L_i(N) \le \varepsilon + \sup V(N)$.
\item
\label{cor:th:HypoConvergence:UniformConvergence}
If $V$ is real-valued on $N$, then $\sup V(N) < \infty$, and
if $V$ is additionally continuous on $N$, then the convergence in
\ref{e:PointwisConvergence} is uniform with respect to $p \in N$.
\qedhere
\end{enumerate}
\end{corollary}

\begin{proof}
If \ref{cor:th:HypoConvergence:Bound}
does not hold, then there exist
$\varepsilon > 0$, $p \in N$ and a sequence
$(x_i)_{i \in \mathbb{N}}$ in $N$ converging to $p$ and
satisfying $L_i(x_i) > \varepsilon + \sup V(N)$
for infinitely many $i \in \mathbb{N}$. This implies
$\limsup_{i \to \infty} L_i(x_i) > V(p)$, which contradicts
Proposition \ref{prop:hypoLimit:i13absoc}.
The same argument with the inequality
$L_i(x_i) > \varepsilon + V(x_i)$ proves the second claim in
\ref{cor:th:HypoConvergence:UniformConvergence}, and the first claim
follows since $V$ is u.s.c. by Corollary
\ref{cor:th:OptimalityPrinciples:MAX}, and so
$V(N) \subseteq \mathbb{R}$ implies $\sup V(N) < \infty$.
Finally, the identity \ref{e:PointwisConvergence} follows from the
estimate $V \le L_i$ and the special case $N = \{ p \}$ of
\ref{cor:th:HypoConvergence:Bound}.
\end{proof}

An interesting special case arises when the cost functions
\ref{e:costfunction} map into the discrete set
\begin{equation}
\label{e:DiscreteCosts}
D = \lambda \mathbb{Z}_{+} \cup \{ \infty \}
\end{equation}
for some $\lambda \in \mathbb{R}_{+}$, in which the subcase
$\lambda = 0$, or equivalently, $D = \{ 0, \infty \}$, corresponds to
qualitative problems.
Then, without loss of generality, all abstract cost functions map into
the set \ref{e:DiscreteCosts} either.
We would like to explicitly spell out this case, which
includes, e.g. the Reach-Avoid Problem and the Minimum Time
Problem in Examples \ref{ex:ReachAvoid} and \ref{ex:MinTime}:

\begin{corollary}
\label{cor:th:QualitativeCompleteness}
Assume \ref{h:Convergence} and \ref{h:ss:AlgorithmicSolution:OverallMethod}.
Suppose that both the concrete cost functions $g$ and $G$
and the abstract cost functions $g_i$ and $G_i$ map into the set
\ref{e:DiscreteCosts}, for some $\lambda \in \mathbb{R}_{+}$ and 
every $i \in \mathbb{N}$.
\\
Then for every compact subset $N \subseteq X$ we have
$\sup L_i(N) = \sup V(N)$ for all sufficiently large
$i \in \mathbb{N}$.
In particular, if $\lambda = 0$ and $V$ vanishes on some compact
subset $N \subseteq X$, so does $L_i$
for all sufficiently large $i \in \mathbb{N}$.
\end{corollary}

Assertion \ref{cor:th:HypoConvergence:Bound} in Corollary
\ref{cor:th:HypoConvergence} and Corollary
\ref{cor:th:QualitativeCompleteness} can be seen as a completeness
results.
Indeed, if for every initial state in a compact
subset $N \subseteq X$ the achievable closed-loop performance for
\ref{e:OCP} is finite, then using sufficiently precise abstractions
it is possible to synthesize controllers for
\ref{e:OCP} whose worst-case performance gaps on $N$ are arbitrarily small.
In  particular,
we obtain controllers to solve qualitative
problems on the whole of $N$ whenever such controllers exist. This is
in contrast with somewhat related results from the literature.
Specifically, there is a method that, given a qualitative control
problem and some perturbation of that problem, returns
either a solution to the former problem in the form of a controller,
or a proof that the latter problem is not solvable
\cite[Cor.~2]{Liu17}. Analogous results for verification problems
appear in \cite{KongGaoChenClarke15}.
While the method does apply to arbitrarily small perturbations,
it is not guaranteed, by the theory in \cite{Liu17},
to ever return a controller
even if the original, unperturbed problem is solvable.

\section{Algorithmic Solution}
\label{s:AlgorithmicSolution}

The practical applicability of our main results in Section
\ref{s:convergence} depends on our ability to both compute finite
abstractions of arbitrary conservatism and solve finite optimal control
problems.
For the sake of self-consistency of the present paper, we shall
discuss both issues, where for the former problem we focus on our
solution in \cite{i17conv} for a class of optimal control problems
arising in the context of sampled-data control systems.
Using e.g.~the method from \cite[Sec.~8.2]{DellnitzJunge02},
it is straightforward to adapt our solution to the simpler case where
the transition function of the plant is given explicitly, rather than
implicitly through sampling a continuous-time system.

\subsection{A sampled optimal control problem}
\label{ss:AlgorithmicSolution:SampledOCP}

We introduce a class of optimal control problems for which we devised an algorithm
in \cite{i17conv} to compute finite abstractions of arbitrary conservatism.
The discrete-time plant represents the sampled behavior of a continuous-time control
system, which we describe by a nonlinear differential equation with additive,
bounded disturbances of the form
\begin{equation}
\label{e:System:c-time}
  \dot x \in f(x,u) + \segcc{-w,w}
\end{equation}
where $f \colon \mathbb{R}^n\times U\to \mathbb{R}^n$, $U\subseteq \mathbb{R}^m$,
and $w\in\mathbb{R}_{+}^n$.
Here, the summation in \ref{e:System:c-time} is interpreted as the Minkowski set
addition \cite{RockafellarWets09}, and $\segcc{-w,w}$ denotes a \concept{hyper interval} in
$\mathbb R^n$ given by
$\segcc{-w,w}=\intcc{-w_1,w_1}\times\ldots \times \intcc{-w_n,w_n}$.
Given an input signal $u \colon J \subseteq{\mathbb{R}} \to U$,
a locally absolutely continuous map
$\xi \colon I \to \mathbb{R}^n$ is a
\concept{solution of \ref{e:System:c-time} on $I$ generated by $u$} if
$I \subseteq J$ is an interval and $\dot \xi(t) \in f(\xi(t),u(t)) + \segcc{-w,w}$ holds
for almost every $t \in I$. Whenever $u$ is constant on $I$ with value $\bar u\in U$, we slightly
abuse the language and refer to $\xi$ as a solution of \ref{e:System:c-time} on $I$
generated by~$\bar u$.

We consider the following optimal control problem associated with the
sampled behavior of~\ref{e:System:c-time}.

\begin{definition}
\label{d:sampledOCP}
Given a sampling time  $\tau > 0$ and cost functions 
\begin{IEEEeqnarray*}{c'c}
  g_1 \colon
  \mathbb{R}^n \times \mathbb{R}^n \times U \to \mathbb{R}_{+} \cup \{ \infty \},
  & 
  G_1 \colon \mathbb{R}^n \to \mathbb{R}_{+} \cup \{ \infty \},
\end{IEEEeqnarray*}
the tuple $\Pi_1=(X_1,U_1,F_1,G_1, g_1)$ is the
\concept{optimal control problem associated with \ref{e:System:c-time}
  and $\tau$}, where
$X_1=\mathbb{R}^n$, $U_1=U$,
and $F_1 \colon X_1 \times U_1 \rightrightarrows X_1$ is implicitly defined by
$x'\in F_1(x,u)$ iff there exists a solution $\xi$ of \ref{e:System:c-time} on $\intcc{0,\tau}$
generated by $u \in U$ that satisfies $\xi(0)=x$ and $\xi(\tau)= x'$.
\end{definition}

The following hypothesis ensures that $\Pi_1$ is actually an
optimal control problem in the sense of Definition~\ref{d:OST} that
additionally satisfies Hypothesis \ref{h:Convergence}, i.e., a problem
to which our results in Section \ref{s:convergence} apply.

\begin{hypothesisA}
  \label{h:sampledOCP}
  The input set satisfies $U=\cup_{i\in \intcc{1;l}} \segcc{\check u_i,\hat u_i}$, with $\check u_i,\hat
  u_i\in\mathbb R^m$, $\check u_i\le \hat u_i$,
  and $l \in \mathbb{N}$.
  The function $G_1$ and $g_1$ is continuous on the set
  $G_1^{-1}(\mathbb{R})$ and $g_1^{-1}(\mathbb R)$, respectively, and
  these sets are open.
  The map $f$ is continuous, and for all $i,j\in \intcc{1;n}$, the partial
derivative $D_jf_i$ with respect to the $j$th
component of the first argument of $f_i$ exists and is continuous.
  Every solution $\xi$
  of~\ref{e:System:c-time} on $\intcc{0,s}$ generated by some $u \in U$,
  where $s < \tau$,
  can be extended to a solution on $\intcc{0,\tau}$ generated by $u$.
\end{hypothesisA}

\begin{lemma}[Lemma 1 \cite{i17conv}]
\label{l:sampledOCP}
Consider an  optimal control problem
$\Pi_1=(X_1,U_1,F_1,G_1,g_1)$ associated
with~\ref{e:System:c-time} and $\tau>0$ and suppose that \ref{h:sampledOCP}
holds. Then $\Pi_1$ is an optimal control problem that
satisfies \ref{h:Convergence}.
\end{lemma}

For the actual computation of abstractions, we introduce
the \concept{domain} $K$ of $(X_1,U_1,F_1,G_1,g_1)$,
\begin{multline}\label{e:domain}
K
=
\{ p \in X_1
\mid
g_1(X_1,p,U_1)\:\cup
\\
g_1(p,X_1,U_1) \cup \{ G_1(p) \}
\neq \{ \infty \}
\}
\end{multline}
which includes the effective domain of the value function.

In the construction of an abstraction of the optimal control problem associated with
\ref{e:System:c-time} various bounds related to the dynamics and the cost
functions are used, as detailed below. Here and in Section \ref{s:Example},
$|x|$ and $\| x \|$ denote the component-wise absolute
value, respectively, the infinity norm of $x \in \mathbb R^n$,
and all balls are understood with respect to the infinity norm.

\begin{hypothesisA}
\label{h:computation}
Let $K$ be defined by~\ref{e:domain}. Let $K'$ be convex and
compact and so that
for every $u\in U$ and every solution $\xi$ of~\ref{e:System:c-time} on
  $\intcc{0,\tau}$ generated by  $u$ with $\xi(0)\in K$ we have $\xi(\intcc{0,\tau})\subseteq K'$.
The constants $A_0\in\mathbb{R}_+^n$, $A_1\in \mathbb{R}^{n\times n}$,
$A_2,A_3 \ge 0$ and $\varepsilon > 0$ satisfy the inequalities  (component-wise)
\begin{subequations}
\label{e:bounds}
\begin{align}
  \label{e:bounds:f}
  A_0&\ge |f(p,u)|+w, \\
  \label{e:bounds:df}
  (A_1)_{i,j}
  &\ge
  \begin{cases}
    D_j f_i(x,u),& \text{if $i=j$,}\\
    | D_j f_i(x,u) |,& \text{otherwise}
  \end{cases}\\
  \intertext{\noindent for all $u\in U$ and all $p\in \bar B(K',\varepsilon)$.
    Moreover, for all $p,\bar p\in G_1^{-1}(\mathbb R)$ we have}
  \label{e:bounds:dG}
  \|p-\bar p\|A_2&\ge|G_1(p)-G_1(\bar p)| , \\
  \intertext{\noindent and for all  $(p,q,u),(\bar p,\bar q,u)\in g_1^{-1}(\mathbb R)$ we have}
  \label{e:bounds:dg}
  (\|p-\bar p\|+\|q-\bar q\|)A_3& \ge
  |g_1(p,q,u)-g_1(\bar p,\bar q,u)|.
\end{align}
\end{subequations}
\end{hypothesisA}
We refer the interested reader to \cite{i17conv} for a discussion of the
computation of the quantities in \ref{h:computation}.

Following \cite{i17conv}, an abstraction $\Pi_2$ of the optimal control problem $\Pi_1$ associated with
\ref{e:System:c-time} and $\tau$ is obtained as
follows. Let $\Pi_1$ and $\Pi_2$ be of the form \ref{e:two:ocp}. The state
alphabet $X_2$ is constructed from a uniform discretization of the domain
\ref{e:domain} of $\Pi_1$ using the discretization parameter $\eta\in
(\mathbb{R}_+\smallsetminus \{0\})^n$. Similarly, the input alphabet $U_2$ is
obtained by a discretization of $U_1$ using the discretization parameter $\mu\in
(\mathbb{R}_+\smallsetminus \{0\})^n$. The transition function $F_2$ is obtained
from an over-approximation of the attainable set of \ref{e:System:c-time} whose
computation is outlined in Algorithm~1 in \cite{i17conv}. 
To this end, the sampling time $\tau$ is subdivided in $k$ inter-sampling times
$t=\tau/k$. At each of those inter-sampling times, the attainable set is
over-approximated by a union of hyper-intervals
using a \concept{growth bound} \cite[Def. VIII.2]{i14sym},
\cite{KapelaZgliczynski09} to bound the
distance of neighboring trajectories. Here the estimates $A_0$ and $A_1$ in \ref{h:computation} are instrumental. 
In order to control the
error due to the over-approximation, at each inter-sampling time, each
hyper-interval in the approximation can be subdivided in smaller
hyper-intervals, whose size is determined by the parameter $\theta>0$.
Throughout the computation, several initial value problems have to be solved
numerically. The resulting error together with other errors, e.g. rounding
errors, can be accounted for using the parameter $\gamma>0$. The cost functions
$G_2$, $g_2$ of the abstraction are derived from the values of the cost
functions $G_1$, $g_1$ evaluated at the
discretized states and inputs. The Lipschitz constants in \ref{e:bounds:dG} and
\ref{e:bounds:dg} are used to ensure that the functions $G_2$, $g_2$ indeed
are upper bounds in the sense of \ref{prop:Abstraction:terminalCost} and
\ref{prop:Abstraction:runningCost} in Proposition~\ref{prop:Abstraction}. The parameters of
the construction of the abstraction in \cite{i17conv} are summarized in
\ref{tab:params}.

\begin{table}[t]
\centering
\caption{Parameters of the computation of the abstraction in \cite{i17conv}.}\label{tab:params}
\begin{tabular}{|l|l|}
\hline
$\eta\in (\mathbb{R}_+\smallsetminus\{0\})^n $ & state alphabet discretization\\
$\mu\in (\mathbb{R}_+\smallsetminus\{0\})^m $ & input alphabet discretization\\
$k\in \mathbb N $ & sample interval discretization\\
$\theta>0 $ & subdivision factor\\
$\gamma>0 $ & bound on numerical errors\\
\hline
\end{tabular}
\end{table}

We use $\Pi$ to refer to the optimal control problem associated with
\ref{e:System:c-time} and $\tau$,
and we consider sequences of
parameters in \ref{tab:params} satisfying
\[
\lim_{i\to \infty} \eta_i = 0,
\lim_{i\to \infty} \mu_i = 0,
\lim_{i\to \infty}
\left(
\theta_i \| \eta_i \| + \frac{1}{k_i} + \gamma_i k_i
\right)
=
0.
\]
Then the method in \cite[Sec.~V]{i17conv} produces a sequence
$(\Pi_i)_{i \in \mathbb{N}}$ of finite
abstractions $\Pi_i$ of some conservatism
$\rho_i \in \mathbb{R}_{+} \cup \{ \infty \}$ of $\Pi$, satisfying
$\lim_{i \to \infty} \rho_i = 0$, as required in hypothesis
\ref{h:ss:AlgorithmicSolution:OverallMethod} in Section
\ref{ss:AlgorithmicSolution:OverallMethod}.
See \cite[Th.~1, 2]{i17conv}.

\subsection{Solution of finite optimal control problems}
\label{ss:AlgorithmicSolution:TheAbstractController}

We propose Algorithm \ref{alg:Dijkstra} to efficiently
solve the optimal control problem \ref{e:OCP} whenever the state and
input alphabets are finite; see Theorem \ref{th:Dijkstra} below. The
algorithm can be regarded as an implementation of the high-level
algorithm in \cite{Knuth77}, with improved run time bound and suitable
modifications to additionally compute a controller realizing the
achievable closed-loop performance. We also present a condition under
which the run time is linear in the size of the abstraction of the
plant. This result applies e.g. to the Reach-Avoid and Minimum Time
Problems in Examples \ref{ex:ReachAvoid} and \ref{ex:MinTime}, and
contains the unweighted case of \cite{DowlingGallier84} as a special
case.
In the following, $\card(M)$ denotes the cardinality of the set $M$.

\begin{algorithm}[t]
  \caption{\label{alg:Dijkstra}Dijkstra-like algorithm to solve finite problems}
  \begin{algorithmic}[1]
    \Input{Optimal control problem $(X,U,F,G,g)$}
    \Require{$X$, $U$ finite}
    \State {$W \defas G$}\hfill{}{// value function}
    \State {$Q \defas \Menge{x \in X}{G(x) < \infty}$\hfill{}{// priority queue}}\label{alg:Dijkstra:Init_Q}
    \State {$E \defas \emptyset$\hfill{}{// set of settled states}}
    \ForAll {$p \in X$}
      \State {$c(p) \defas \emptyset$}\hfill{}{// controller}
    \EndFor
    \While {$Q \not= \emptyset$}\label{alg:Dijkstra:while_entry}
      \State {$q \defin \argmin \Menge{ W(x) }{ x \in Q }$}\label{alg:Dijkstra:pick_q}
      \State {$Q \defas Q \setminus \{ q \}$}\label{alg:Dijkstra:Remove_q_from_Q}
      \State {$E \defas E \cup \{ q \}$}\label{alg:Dijkstra:Add_q_to_E}
      \ForAll {$(p,u) \in F^{-1}(q)$}\label{alg:Dijkstra:for_all_predecessors_entry}
        \State {$M \defas \max \Menge{ g(p,y,u) + W(y) }{ y \in F(p,u)}$}\label{alg:Dijkstra:Zuweisung_M}
        \If {$F(p,u) \subseteq E$ and $W(p) > M$}\label{alg:Dijkstra:if_condition}
          \State {$W(p) \defas M$}\label{alg:Dijkstra:Zuweisung_W(p)}
          \State {$Q \defas Q \cup \{ p \}$}\label{alg:Dijkstra:Add_q_to_Q}
          \State {$c(p) \defas \{ u \}$}\label{alg:Dijkstra:Zuweisung_mu(p)}
        \EndIf
      \EndFor
    \EndWhile
    \Output {$c$, $W$}
  \end{algorithmic}
\end{algorithm}

\begin{theorem}
\label{th:Dijkstra}
Let \ref{e:OCP} be an optimal control problem with finite $X$ and
$U$. Then Algorithm \ref{alg:Dijkstra} terminates.
\\
Suppose that the maps $c$ and $W$ are returned on termination, and
let $C = ( Z, Z, X, U, H )$, where $Z$ is any singleton
set, $u_0 \in U$, and
$H \colon Z \times X \rightrightarrows Z \times U \times \{0,1\}$ is
given by
\begin{equation}
\label{e:th:Dijkstra:H'}
H( Z, p )
=
\begin{cases}
Z \times \{ u_0 \} \times \{ 1 \}, &\text{if $c(p) = \emptyset$,}\\
Z \times c(p) \times \{ 0 \},      &\text{otherwise.}
\end{cases}
\end{equation}
Then $C$ is a static
controller for $S$,
and $L = V = W$, where
$S$, $L$ and $V$ denote the system \ref{e:def:System},
the closed-loop value function of \ref{e:OCP} associated with
$C$, and the value function of \ref{e:OCP}.
\\
Moreover, Algorithm \ref{alg:Dijkstra} can be implemented
such that it runs in $O(m + n \log n)$ time,
where
$n = \card(X)$ and $m = \sum_{p \in X} \sum_{u \in U} \card(F(p,u))$,
and in $O(m)$ time if additionally
\begin{equation}
\label{e:cor:th:Dijkstra}
g(X,X,U) \subseteq \{ \gamma, \infty \}
\text{\ and\ }
G(X) \subseteq \{ \Gamma, \gamma + \Gamma, \infty \}
\end{equation}
for some $\gamma, \Gamma \in \mathbb{R}_{+}$.
\end{theorem}

\begin{proof}
Observe that $M \ge W(q)$, and in turn, $p \not= q$, on lines
\ref{alg:Dijkstra:Zuweisung_W(p)}-\ref{alg:Dijkstra:Zuweisung_mu(p)}.
Thus, throughout the algorithm on lines
\ref{alg:Dijkstra:Remove_q_from_Q}-\ref{alg:Dijkstra:Zuweisung_mu(p)},
the value of $W(q)$
monotonically increases and $W(q) \ge \max W(E)$. Then
$p \not\in E$ on lines
\ref{alg:Dijkstra:Zuweisung_W(p)}-\ref{alg:Dijkstra:Zuweisung_mu(p)},
and so
each $q$ is removed from $Q$ at most once. This shows that the
\verb|while|-loop on lines
\ref{alg:Dijkstra:while_entry}-\ref{alg:Dijkstra:Zuweisung_mu(p)} is
entered at most $n$ times. Moreover, $F^{-1}(q) \subseteq X \times U$
on line \ref{alg:Dijkstra:for_all_predecessors_entry}, and so the
algorithm terminates as $X \times U$ is finite.

Next note that $W$ is a monotonically decreasing sequence of
functions $X \to \intcc{0,\infty}$ bounded above by $G$.
Using induction we see that
$Q \cup E = W^{-1}(\mathbb{R}_{+})$
on line
\ref{alg:Dijkstra:Zuweisung_mu(p)}.

If $W \ge V$, then $M \ge P(V)(p)$ on line
\ref{alg:Dijkstra:Zuweisung_W(p)}, where $P$ is the dynamic
programming operator associated with \ref{e:OCP}, and so $W \ge V$ on
lines
\ref{alg:Dijkstra:while_entry}-\ref{alg:Dijkstra:Zuweisung_mu(p)}
throughout the algorithm, as $V = P(V)$ by Th.~\ref{th:OptimalityPrinciple}.
We claim that $W \le P(W)$ upon termination,
which implies $W = V$ by Th.~\ref{th:OptimalityPrinciple}. Assume
the contrary. Then, as $W \le G$, there exist
$(p,u) \in X \times U$
such that
\begin{equation}
\label{e:th:Dijkstra:proof:1}
W(p)
>
\max \Menge{ g(p,y,u) + W(y) }{ y \in F(p,u)},
\end{equation}
and in turn, $F(p,u) \subseteq E$ since
$E = W^{-1}(\mathbb{R}_{+})$.
Let $q \in F(p,u)$ be the element that is last
added to $E$. Then, upon its addition on line
\ref{alg:Dijkstra:Add_q_to_E} we have $W(p) > M \ge W(q)$ on line
\ref{alg:Dijkstra:if_condition} by \ref{e:th:Dijkstra:proof:1}. Thus,
line \ref{alg:Dijkstra:Zuweisung_W(p)} is executed, which contradicts
\ref{e:th:Dijkstra:proof:1} and so implies that $W = V$ upon
termination.

Obviously, $C$ is a static controller for $S$.
To show that $L = W$ upon termination,
first suppose that
$q \not\in E$ upon termination.
Then $W(q) = \infty$ and $c(q) = \emptyset$, and so $L(q) = \infty$ by
\ref{e:th:Dijkstra:H'}. Hence, it suffices to show that
$L(q) = W(q)$ on line \ref{alg:Dijkstra:Add_q_to_E} throughout the
algorithm.
To this end, we proceed by induction and assume that $L(x) = W(x)$
holds on line \ref{alg:Dijkstra:Remove_q_from_Q} for all
$x \in E$. Note that $c(q) \not= \emptyset$
since line \ref{alg:Dijkstra:Add_q_to_Q} must have been executed at
least once, and additionally
\begin{equation}
\label{e:th:Dijkstra:proof:2}
W(q)
=
\max \Menge{ g(q,y,c(q)) + L(y) }{ y \in F(q,c(q))}.
\end{equation}
Then $v(0) = 0$ for every
$(u,v,x) \in \mathcal{B}_q(C \times S)$ by
\ref{e:th:Dijkstra:H'}, and in turn,
$J(u,v,x)
=
g(q,x(1),c(q))
+
J(\sigma u, \sigma v, \sigma x)$. Then $L(q) = W(q)$ by
\ref{e:th:Dijkstra:proof:2}.

The data $G$, $W$, $E$ and $c$ are maintained as arrays, so
the respective operations in the algorithm require unit time. Given
an adjacency lists representation \cite{AhujaMagnantiOrlin93} of $F$
that also stores the map $g$, both an analogous representation
of $F^{-1}$ can be obtained and the condition \ref{e:cor:th:Dijkstra}
can be verified, in $O(m)$ time.

Lines
\ref{alg:Dijkstra:Zuweisung_M}-\ref{alg:Dijkstra:Zuweisung_mu(p)} are
executed at most $m$ times. Using auxiliary counters
the tests $F(p,u) \subseteq E$ on line
\ref{alg:Dijkstra:if_condition} take $O(m)$ total time
\cite{DowlingGallier84}, and analogously for computing the
maximum on line \ref{alg:Dijkstra:Zuweisung_M}.
Thus, Algorithm \ref{alg:Dijkstra} requires $O(m)$ time, plus
the time for executing line \ref{alg:Dijkstra:Init_Q}, executing lines
\ref{alg:Dijkstra:pick_q}, \ref{alg:Dijkstra:Remove_q_from_Q} and
\ref{alg:Dijkstra:Add_q_to_Q} at most $n$ times, and for executing
line \ref{alg:Dijkstra:Zuweisung_W(p)} at most $m$ times.
Consequently, the first time bound is met if
$Q$ is maintained as a Fibonacci heap
\cite{AhujaMagnantiOrlin93}.
If condition \ref{e:cor:th:Dijkstra} holds,
then $M = \gamma + W(q)$ on line \ref{alg:Dijkstra:Zuweisung_W(p)},
and so $W(Q) \subseteq \{ W(q), \gamma + W(q) \}$ on line
\ref{alg:Dijkstra:Remove_q_from_Q}.
Thus, the second bound is met if $Q$ is maintained as a FIFO queue
\cite{AhujaMagnantiOrlin93}.
\end{proof}

\subsection{Comments on Computational Complexity}
\label{ss:AlgorithmicSolution:ComputationalComplexity}

In our approach, the concrete control problem \ref{e:OCP} is
discretized first, resulting in an abstration which is
solved subsequently. Bounds on the computational complexity have
been provided in Theorem \ref{th:Dijkstra} for the second step, and in
\cite[Sec.~III.D]{i11abs}, for the special case of the first step
when $k=1$ and $\eta$, $\mu$, $\Theta$ and $\gamma$ are constants.
The estimates show that the overall computational effort is enormous
and has to be expected to grow rapidly with the dimension of $X$, and
even more acutely so when a sequence of abstractions of decreasing
conservatism is to be computed.
While the problem is found with all discretization based methods to
solve \ref{e:OCP}, several strategies to somewhat relieve the
computational burden that have been proposed,
e.g.~\cite{MunosMoore02,RunggerStursberg12,i15grid}, could
potentially be extended to our setting.

\section{Illustrative Examples and Applications}
\label{s:Example}

We shall demonstrate our approach on three
optimal control problems. In every of these three cases, and in contrast
to the theory presented in this paper, none of the related works
discussed in Section \ref{s:intro} is capable of synthesizing
controllers together with upper bounds on their performances that
arbitrarily closely approximate the best achievable performance.

\subsection{A minimum time problem involving chaotic dynamics}
\label{ss:examples:logistic_map}

To demonstrate the capability of our theory to approximate
complex value functions, we first apply it to an instance $\Pi$
of the Minimum Time Problem in Example \ref{ex:MinTime} whose
underlying dynamics is chaotic. Specifically,
$\Pi = (\intcc{0,1},\{0\}, F, G, g)$, where
the transition function $F$ is the \concept{logistic map}
\cite{Devaney89},
$F(p,0)=\{4p(1-p)\}$, and the
target and obstacle sets are given by
$D=\intoo{0.415,0.69}$
and $M = \emptyset$.

The value function $V$ of $\Pi$ is discontinuous and rather irregular, see
\ref{f:logicsticmap}, but can be determined exactly by rewriting the
iteration in Corollary \ref{cor:th:OptimalityPrinciples:MAX} into an
iteration for sublevel sets,
$V^{-1}(0) = D$ and
$V^{-1}(\intcc{0;T+1}) = F(\cdot,0)^{-1}(V^{-1}(\intcc{0;T}))$.
\begin{figure}
\centering
\includegraphics[width=.499\linewidth]{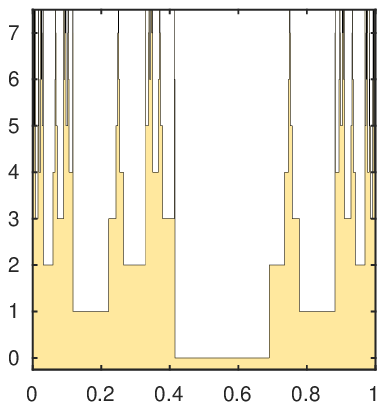}%
\hspace*{\fill}
\includegraphics[width=.499\linewidth]{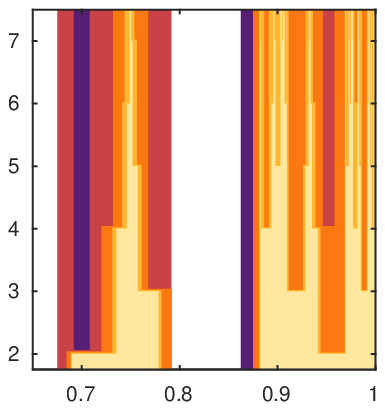}
\caption{\label{f:logicsticmap}
Minimum time problem involving chaotic dynamics.
Left: Hypograph of the value function $V$. Right: Hypographs of
$V$ (light yellow) and of the approximate value functions $V_{N}$ for
$N \in \{ 40, 60, 85, 400 \}$
(purple, red, orange and dark yellow, respectively).
}%
\end{figure}
For everyy $N \in \mathbb{N}$ it is straightforward to compute an
abstraction
$\Pi_{N}
=
( X_{N}, \{0\}, F_{N}, G_{N}, g_{N})
$
of conservatism $1 / N$ of $\Pi$, where
$F_N$ satisfies the conditions in Def.~\ref{def:AbstractionOfPrecision},
\begin{align*}
 X_{N}
&=
\{ \Omega_0, \ldots, \Omega_{N} \},\\
\Omega_i
&=
\left(
\tfrac{i}{N}+\intcc{-\tfrac{1}{2N},\tfrac{1}{2N}}
\right)
\cap
\intcc{0,1},
\\
 g_{N}(\Omega,\Omega',0)
&=
1,
\text{\ and}
\\
 G_{N}(\Omega)
&=
\begin{cases}
0, &\text{if } \Omega \subseteq D,\\
\infty, &\text{otherwise,} 
\end{cases}
\end{align*}
for all $i \in \intcc{0;N}$ and all
$\Omega, \Omega' \in  X_N$.
The value function $V_N$ of $\Pi_N$ is easily computed using Algorithm
\ref{alg:Dijkstra} in Section \ref{ss:AlgorithmicSolution:TheAbstractController}.
\ref{f:logicsticmap} illustrates the approximation of $V$ by $V_N$,
for selected values of the conservatism $1 / N$.

\subsection{An entry-time problem for the inverted pendulum}
\label{ss:examples:InvertedPendulum}

We consider a variant of the popular inverted pendulum problem with
perturbations, where the motion of the cart is not modeled;
see e.g.
\cite{i11abs,FantoniLozano02}.
The acceleration $u$ of the cart, which is constrained to
$\intcc{-2,2}$, is the input to the system%
\begin{subequations}
\label{e:ex:cart}
\begin{align}
\label{e:ex:cart:x1}
\dot x_1 &=   x_2\\
\label{e:ex:cart:x2}
\dot x_2 &\in \sin(x_1)+u\cos(x_1)-2\kappa x_2+\intcc{-w,w},
\end{align}
\end{subequations}
the states $x_1$ and $x_2$
correspond to the angle, respectively, the angular velocity of the pole,
$\kappa=0.01$ is a friction coefficient, and
$w=0.1$ accounts for any uncertainties.

We restrict the domain of the problem to $K=\intoo{-2\pi,2\pi}\times\intoo{-3,3}$,
i.e., the set $\mathbb{R}^2 \setminus K$ is an obstacle,
and choose a neighborhood $D$ of the upwards pointing equilibrium
$(0,0)$,
\begin{align*}
D&=\Menge{x \in \mathbb R^2}{63 x^2_1 + 12 x_2 x_1 + 56 x^2_2 < 42},
\end{align*}
as the target set.
In correspondence with $K$ and $D$, we
define the terminal and running cost functions $G$ and $g$ by
$G(p)=0$ if $p \in D \cap K = D$, $G(p)=\infty$, otherwise, and
\[
g(p,q,u)=
\begin{cases}
u^2, &\text{if } p \in K,\\
\infty, &\text{otherwise}.
\end{cases}
\]
We use $\Pi$ to refer to the optimal control problem associated
with the system \ref{e:ex:cart}, the sampling time $\tau = 0.2$ and
the cost functions $G$ and $g$. With $\Pi$
we aim at minimizing the actuation
energy to steer the system into the target $D$.
We pick the constants in \ref{h:computation} to
\begin{align*}
A_0 \defas \begin{pmatrix} 4\\ 2.5\end{pmatrix},\;
A_1 \defas \begin{pmatrix} 0 & 1\\ 2.25 & -0.02 \end{pmatrix} ,\;
A_2 \defas 0,\; A_3 \defas 0.
\end{align*}
We use $A_0$ to verify that $K'=\cBall(\mathrm{cl}K,0.9)$
contains any solution of \ref{e:ex:cart} originating form $K$ since
$\cBall(K,\tau\|A_0\|)\subseteq K'$.
Moreover, \ref{e:bounds} is satisfied
on $\intcc{-8,8}\times \intcc{-4,4}\supseteq\cBall(K',0.1) $,
and we see that \ref{h:computation} holds for $\varepsilon=0.1$.

We conducted several experiments
using $\theta=1$ and four parameter tuples
$(\eta,\mu,k)$ with values 
$p_1=((0.08,0.08),0.2,1)$, $p_2=((0.04,0.04),0.15,2)$,
$p_3=((0.02,0.02),0.1,3)$ and
$p_4=((0.01,0.01),0.05,4)$.
For the solution of initial value problems, which are necessary in the
construction of the abstraction \cite{i17conv}, we use the Taylor series
method~\cite{NedialkovJacksonCorliss99} of order
$5$ with stepsize $\tau/(5k)$. We use $\gamma$ to account for any
numerical errors, which we
derive from the $6$th order remainder term of the Taylor expansion
maximized over the appropriate domain. Specifically, for 
$k=1$,
$k=2$,
$k=3$ and
$k=4$ we obtain
$\gamma=6.3\cdot 10^{-7}$,
$\gamma=9.9\cdot 10^{-9}$,
$\gamma=8.7\cdot 10^{-10}$
and 
$\gamma=1.6\cdot 10^{-10}$,
respectively. 
The computation time to compute the abstraction $\Pi_i$ and
the optimal controller $C_i$ (Alg.~\ref{alg:Dijkstra})
is $0.5$, $8.5$, $139$ and $4715$ seconds,
for the parameter tuple $p_i$, $i \in \intcc{1;4}$, respectively.
(Here and for the following example, computations are conducted on 3.5 GHz Intel Core i7 CPU with 32GB memory.)
The performance of the controllers $C_1\circ{\in}$ through
$C_4\circ{\in}$ is illustrated in \ref{f:cartpole}.

\begin{figure}
  \centering
  \includegraphics[width=.499\linewidth]{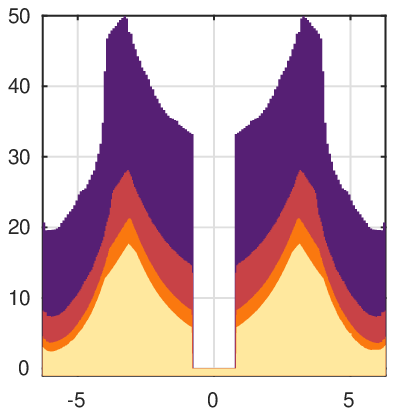}%
  \includegraphics[width=.499\linewidth]{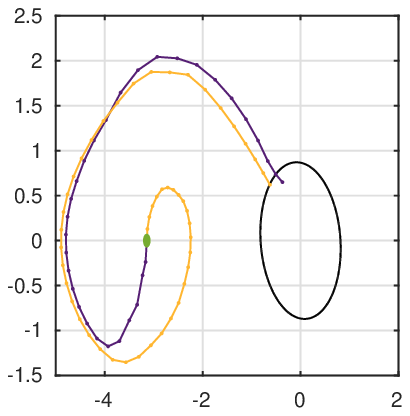}
  \caption{Entry-time problem for the inverted pendulum.
  Left: Cross-section of the hypograph of the closed-loop value function
  of $\Pi$ associated with the controller $C_i\circ{\in}$, ranging over
  $x_1 \in \intcc{-2\pi,2\pi}$ for fixed $x_2=0$, for
  $i\in\{1,2,3,4\}$ (purple, red, orange and yellow, respectively).
  Right: Closed-loop trajectories generated by the controllers
  $C_1\circ{\in}$ (purple) and $C_4\circ{\in}$ (yellow);
  the closed-loop value function at the initial states is bounded by
  $47.68$, respectively, $17.65$.
  The initial position is marked by the green dot and the
  target set $D$ is illustrated by the black ellipse.
}\label{f:cartpole}
\end{figure}

\subsection{The Homicidal Chauffeur Game}
\label{ss:examples:HomicidalChauffeur}

In this pursuit-evasion game, a car with restricted turning radius,
traveling at some constant velocity, aims at catching an agile pedestrian
as quickly as possible \cite{Isaacs65}. 
The problem can be posed as a Minimum Time Problem by
choosing the center of the
car as origin and directing the $y$ axis along the velocity vector of
the car.
The dynamics is then described by
\begin{align*}
\dot x &= -y u + v_1\\
\dot y &= xu -1 + v_2,
\end{align*}
where the input $|u|\le 1$ is the forward velocity of the car,
and $v=(v_1,v_2)$ is the velocity vector of the pedestrian
\cite{Isaacs65}, which we
consider as a perturbation with bound $\| v \| \le 0.3$.
Using the sampling time
$\tau=0.1$ and the domain $K=\intoo{-5,5}\times \intoo{-5,5}$,
we cast the sampled differential game as Minimum Time Problem with
target set $D = \Menge{(x,y) \in \mathbb R^2}{x^2 + y^2 < 0.9}$ and
the obstacle set $M=\mathbb R^2\smallsetminus K$. The cost functions
follow according to Example~\ref{ex:MinTime} and it is straightforward
to verify the Hypothesis \ref{h:computation} as follows. We fix
$\varepsilon=0.1$,  $A_0=(6.4,6.4)$, and $(A_1)_{11}=(A_1)_{22}=0$,
$(A_1)_{12}=(A_1)_{21}=1$, $A_2=A_3=0$ and
$K' = \intcc{-6,6} \times \intcc{-6,6}$. The estimates
\ref{e:bounds} are obvious, and \ref{e:bounds:f} implies
that every solution $\xi$ on $\intcc{0,\tau}$ evolves inside
$\cBall(K,\tau \|A_0\|) \subseteq K'$, and so
\ref{h:computation} holds.

We approximately solve $\Pi$ using $\theta=2$ and four parameter tuples
$(\eta,\mu,k)$ with values
$p_1=((0.03,0.03),0.2,1)$, 
$p_2=((0.02,0.02),0.1,2)$,
$p_3=((0.015,0.015),0.1,3)$ and
$p_4=((0.01,0.01),0.05,4)$.
As the nominal dynamics
under constant control inputs can be solved exactly, we neglect the
numerical errors and set $\gamma=0$.
The computation time to compute the abstraction $\Pi_i$ and
the optimal controller $C_i$ (Alg.~\ref{alg:Dijkstra})
is $3.5$, $34$, $133$ and $1851$ seconds,
for the parameter tuple $p_i$, $i \in \intcc{1;4}$, respectively.
Naturally, with finer discretization parameters the computation times increases.
The performance of the controllers $C_1\circ{\in}$ through
$C_4\circ{\in}$ is illustrated in \ref{f:homchauff}.

\begin{figure}
  \centering
  \includegraphics[width=.499\linewidth]{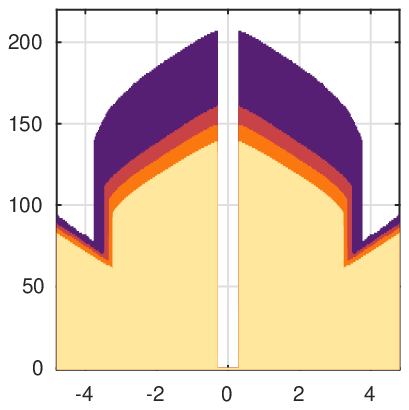}%
  \includegraphics[width=.499\linewidth]{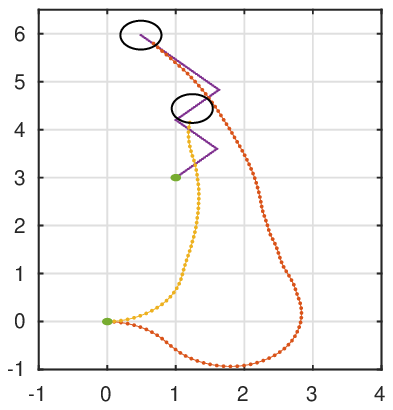}
  \caption{Homicidal Chauffeur Game.
  Left: Cross-section of the hypograph of the closed-loop value
  function of $\Pi$ associated with $C_i\circ{\in}$, ranging over
  $x\in\intcc{-4.5,4.5}$ for fixed $y=0$, for  $i\in\{1,2,3,4\}$ (purple,
  red, orange and yellow, respectively).
  Right: Simulation of the closed-loop. The position of the pedestrian (evader)
  is illustrated in purple. The position of the car (pursuer) for $C_1\circ{\in}$ and
  $C_4\circ{\in}$ is shown in red,
  respectively, yellow.
  The initial positions are marked by the green squares, and
  the capture radius $0.3$ is indicated by the black circles.  
  The worst-case capture times from the initial state for $C_1\circ \in$ and
  $C_4\circ \in$ are bounded by
  $17.0$, respectively, $5.3$ seconds.
}
\label{f:homchauff}
\end{figure}

\section{Summary and Conclusions}
\label{s:Conclusions}

We have presented a novel approach to solve a class of leavable, undiscounted optimal control problems
in the minimax sense for nonlinear control systems in the presence of perturbations and constraints.
The approach is correct-by-construction, i.e., the closed-loop value
function of the synthesized controller is upper bounded by the
closed-loop value function of the abstract controller. Compared to
previously known results, our approach is applicable to more general
cost functions and plant dynamics, and the resulting controllers are
memoryless and symbolic.
Moreover, as we have shown, the
closed-loop value function associated with the concrete controller
hypo-converges to the concrete value function as the conservatism of the
abstraction approaches zero.
This
powerful convergence result distinguishes itself form previously known results
in several important aspects. Most notably, it applies to discontinuous value
functions and implies that our approach is complete in a well-defined
sense.
We have illustrated our results on three optimal control problems,
two of which involving discrete-time plants that represent the
sampled behavior of continuous-time, nonlinear control systems with
additive disturbances. Here, we employed an algorithm that we have
proposed in \cite{i17conv}, to compute abstractions of arbitrary
conservatism.
To increase the computational efficiency of the overall
synthesis approach proposed in this paper is a subject of our current
research.

\appendix

\subsection{The Notion of Hypo-Convergence}
\label{appx:HypoConvergence}

The result below shows that, for the special case considered in Definition
\ref{def:hypoLimit:i13absoc}, that definition is equivalent to
respective definitions in the literature
\cite[Ch.~7.B]{RockafellarWets09},
\cite[Cor.~VII.5.26]{HuPapageorgiou97.i}.
\begin{proposition}
\label{prop:hypoLimit:i13absoc}
Let $X$, $V$ and $L$ be as in Definition
\ref{def:hypoLimit:i13absoc}. Then $V = \hypolim_{i \to \infty} L_i$
iff
$\limsup_{i \to \infty} L_i (x_i) \le V(p)$
for every $p \in X$ and every sequence $(x_i)_{i \in \mathbb{N}}$
converging to $p$.
\end{proposition}

\begin{proof}
For sufficiency, let $p \in X$ and $\varepsilon > 0$, and assume that
the condition in Definition \ref{def:hypoLimit:i13absoc} does not
hold. Then there exists a sequence $(x_i)_{i \in \mathbb{N}}$ in $X$
converging to $p$ and satisfying $L_i(x_i) > V(p) + \varepsilon / 2$
for infinitely many $i \in \mathbb{N}$. This implies
$
\limsup_{i \to \infty} L_i (x_i)
>
V(p)
$,
which is a contradiction.
For necessity, assume that the latter inequality holds for some
$p \in X$ and some sequence $(x_i)_{i \in \mathbb{N}}$ in $X$
converging to $p$.
Then $L_i(x_i) \ge \lambda > V(p)$ for some $\lambda \in \mathbb{R}$
and infinitely many $i \in \mathbb{N}$.
In addition, as $V$ is u.s.c., there exists $\varepsilon > 0$ such
that $V(q) < \lambda - \varepsilon$ for all
$q \in \oBall(p,2 \varepsilon)$.
As $V = \hypolim_{i \to \infty} L_i$
there exists a neighborhood $N \subseteq X$ of $p$ such that
\ref{e:def:hypoLimit:i13absoc} holds for all sufficiently large
$i \in \mathbb{N}$. Then there exists some $i$ such that
$x_i \in \oBall(p,\varepsilon)$ and
$(x_i,\lambda) \in \oBall(\hypo V,\varepsilon)$.
In turn, there exists $(q,\alpha) \in \hypo V$ such that
$\lambda < \alpha + \varepsilon$
and $x_i \in \oBall(q,\varepsilon)$. This implies
$q \in \oBall(p,2 \varepsilon)$, hence
$\alpha \le V(q) < \lambda - \varepsilon$, which is a contradiction.
\end{proof}

\subsection{Some Results on Semi-Continuous Maps}
\label{appx:maths}

\noindent
Throughout, $X$ and $Y$ are metric spaces.
See \cite{HuPapageorgiou97.i,KosmolMullerWichards11}.

\begin{theorem}[Berge's Maximum Theorem]
\label{th:BergesMaximumTheorem}
Let $H \colon X \rightrightarrows Y$ be compact-valued and u.s.c.,
and let $f \colon X \times Y \to \intcc{ -\infty, \infty }$ be u.s.c..
Then the map $g \colon X \to \intcc{ -\infty, \infty }$ defined by
$g(x) = \sup \Menge{f(x,y)}{y \in H(x)}$ is u.s.c..
\end{theorem}

\begin{proposition}
\label{prop:usc_stability_of_monotone_convergence}
Let $\Omega \subseteq X$ be compact, and suppose that the sequence
$(f_k)_{k \in \mathbb{N}}$ of u.s.c.~maps
$f_k \colon X \to \intcc{-\infty,\infty}$ is monotonically decreasing
and converges pointwise to
$g \colon X \to \intcc{-\infty,\infty}$. Then
$
\lim_{k \to \infty} \sup_{x \in \Omega} f_k( x )
=
\sup_{x \in \Omega} g(x)
$,
where limits are understood to take values in
$\intcc{ -\infty, \infty }$.
\end{proposition}

\begin{proposition}
\label{prop:uscCompact}
The map $H \colon X \rightrightarrows Y$ is compact-valued and
u.s.c.~iff the following condition holds:\\
If $(x_k,y_k)_{k \in \mathbb{N}}$ is a sequence in the graph of $H$
and $(x_k)_{k \in \mathbb{N}}$ converges to $p \in X$, then there exists a
subsequence of $(y_k)_{k \in \mathbb{N}}$ converging to some point in
$H(p)$.
\end{proposition}
\begin{corollary}
\label{cor:uscCompact}
If the map $H \colon X \rightrightarrows Y$ is compact-valued and
u.s.c., then $H(\Omega)$ is compact for all compact subsets
$\Omega \subseteq X$.
\end{corollary}

\bibliographystyle{IEEEtran}
\bibliography{MR/cdc13,GR/IEEEtranBSTCTL,GR/preambles,GR/mrabbrev,GR/strings,GR/fremde,GR/eigeneCONF,GR/eigeneJOURNALS,GR/eigenePATENT,GR/eigeneREPORTS,GR/eigeneTALKS,GR/eigeneTHESES}
\ifx\arxivVersion\undefined%
\def\extension{\ifpdf pdf\else eps\fi}
\IfFileExists{GR/greissig_bio}{%
\IfFileExists{GR/greissig2011.\extension}{%
\begin{IEEEbiography}%
[{\includegraphics[width=1in,height=1.25in,clip,keepaspectratio]{GR/greissig2011}}]%
{Gunther Reissig}
\input{GR/greissig_bio}
\end{IEEEbiography}
}{%
\begin{IEEEbiographynophoto}{Gunther Reissig}
\input{GR/greissig_bio}
\end{IEEEbiographynophoto}
}%
}{}
\IfFileExists{MR/mrungger_bio}{%
\IfFileExists{MR/mrungger.\extension}{%
\begin{IEEEbiography}%
[{\includegraphics[width=1in,height=1.25in,clip,keepaspectratio]{MR/mrungger}}]%
{Matthias Rungger}
\input{MR/mrungger_bio}
\end{IEEEbiography}
}{%
\begin{IEEEbiographynophoto}{Matthias Rungger}
\input{MR/mrungger_bio}
\end{IEEEbiographynophoto}
}%
}{}
\makeatletter
\global\let\extension\@undefined
\makeatother
\fi
\ifx\DraftVersion\undefined\relax\else%
\IfFileExists{ReplyToReviewers.tex}{%
\cleardoublepage
\ifCLASSOPTIONtwocolumn\noindent\onecolumn\large\fi
\include{ReplyToReviewers}
}{}%
\fi
\end{document}